\documentclass[11pt,reqno,a4paper]{amsart}

\usepackage[T1]{fontenc}
\usepackage[utf8]{inputenc}

\usepackage{libertinus}
\usepackage{microtype}

\usepackage{geometry}

\usepackage{xcolor}

\usepackage{amsmath,amssymb,amsthm,amsfonts}
\usepackage{mathtools}
\usepackage{mathrsfs}

\usepackage{graphicx}

\usepackage[
colorlinks=true,
linkcolor=blue,
citecolor=blue,
urlcolor=blue
]{hyperref}

\usepackage[nameinlink,noabbrev]{cleveref}

\usepackage{titlesec}

\titleformat{\section}
  {\normalfont\large\bfseries\centering}
  {\thesection.}
  {0.75em}
  {}

\titlespacing*{\section}
  {0pt}
  {3.2ex plus 1ex minus .2ex}
  {1.6ex plus .3ex}

\titleformat{\subsection}
  {\normalfont\normalsize\bfseries}
  {\thesubsection.}
  {0.75em}
  {}

\titlespacing*{\subsection}
  {0pt}
  {2.4ex plus .8ex minus .2ex}
  {1.0ex plus .2ex}

\titleformat{\subsubsection}
  {\normalfont\normalsize\bfseries\itshape}
  {\thesubsubsection.}
  {0.75em}
  {}

\titlespacing*{\subsubsection}
  {0pt}
  {2.0ex plus .6ex minus .2ex}
  {0.7ex plus .2ex}

\usepackage{etoolbox}

\makeatletter
\AtBeginEnvironment{thebibliography}{%
  \def\@mklab#1{#1\hfil}%
}
\makeatother
  
\numberwithin{equation}{section}

\usepackage{aliascnt}

\newtheoremstyle{compact}
  {6pt}   
  {1 pt}   
  {\itshape}
  {}
  {\bfseries}
  {.}
  {0.5em}
  {}

\theoremstyle{compact}

\newtheorem{theorem}{Theorem}[section]

\newaliascnt{proposition}{theorem}
\newtheorem{proposition}[proposition]{Proposition}
\aliascntresetthe{proposition}

\newaliascnt{lemma}{theorem}
\newtheorem{lemma}[lemma]{Lemma}
\aliascntresetthe{lemma}

\newaliascnt{corollary}{theorem}
\newtheorem{corollary}[corollary]{Corollary}
\aliascntresetthe{corollary}

\theoremstyle{definition}
\newaliascnt{definition}{theorem}
\newtheorem{definition}[definition]{Definition}
\aliascntresetthe{definition}

\newtheorem*{question}{Question}

\newtheorem*{example}{Example}

\theoremstyle{remark}
\newaliascnt{remark}{theorem}

\theoremstyle{definition}
\newtheorem{remark}[theorem]{Remark}

\aliascntresetthe{remark}

\usepackage{etoolbox}

\AtBeginEnvironment{proof}{\vspace{6pt}}

\usepackage{enumitem}

\setlist{
itemsep=2pt,
topsep=4pt
}

\usepackage[toc,page]{appendix}

\newcommand\subsetsim{\mathrel{%
\ooalign{\raise0.2ex\hbox{$\subset$}\cr\hidewidth\raise-0.8ex\hbox{\scalebox{0.9}{$\sim$}}\hidewidth\cr}}}

\usepackage{comment}

\makeatletter
\renewcommand{\l@section}[2]{%
  \par\addpenalty\@secpenalty
  \addvspace{1.0em plus 1pt}%
  \@tempdima 1.5em
  \begingroup
    \parindent \z@
    \rightskip \@pnumwidth
    \parfillskip -\@pnumwidth
    \leavevmode
    \bfseries
    \advance\leftskip\@tempdima
    \hskip -\leftskip
    #1\nobreak\hfil \nobreak\hb@xt@\@pnumwidth{\hss #2}\par
  \endgroup}
\makeatother

\usepackage{bbm}
\usepackage{amscd}
\usepackage[all,cmtip]{xy}
\usepackage{changepage}
\usepackage{calc}
\usepackage{scalerel}
\usepackage{extarrows}
\usepackage{aliascnt}

\DeclareMathSymbol{\shortminus}{\mathbin}{AMSa}{"39}

\newcommand{\ignore}[1]{}

\begin{document}

\title{Stealthy point processes and lattice induction}

\author[M. Bj\"orklund]{Michael Bj\"orklund}
\address{Department of Mathematics, Chalmers University of Technology and University of Gothenburg, Gothenburg, Sweden}
\email{micbjo@chalmers.se}



\begin{abstract}
We prove a realization theorem for stealthy point processes
based on lattice induction, together with a converse in
dimension one. Every probability-preserving
$\mathbb R^d$-action induced from a full-rank lattice admits a
generating Delone cross-section whose Bartlett spectrum
vanishes on a neighborhood of the origin. The construction can
be chosen so that first-order linear statistics detect a
nonzero part of the inducing spectral type. Bernoulli bases
yield a nonzero absolutely continuous component, while weakly
mixing bases of singular maximal spectral type yield a nonzero
singular-continuous component. To our knowledge, these are the
first rigorously constructed translation-invariant stealthy
point processes on $\mathbb R^d$ with non-pure-point Bartlett
spectrum.

Conversely, let $\eta$ be an ergodic translation-invariant
point process on $\mathbb R$ with positive intensity and local
second moments. If
\[
    \int_{0<|\xi|<1}\frac{1}{\xi^2}\,d\sigma_\eta(\xi)<\infty,
\]
then its translation action has a nonzero eigenvalue and is
lattice-induced. Consequently, an ergodic
probability-preserving Borel $\mathbb R$-space is
lattice-induced if and only if it admits a generating stealthy
Delone cross-section.

This should be compared with a theorem of Borichev, Sodin and
Weiss stating that a translation-invariant point process on
$\mathbb Z$ with proper spectral support is periodic.
\end{abstract}

\maketitle

\section{Introduction}

Rigorous examples of stealthy point processes have so far been
largely crystalline. We show that this is not a general feature.
Every probability-preserving $\mathbb R^d$-action induced from a
full-rank lattice admits a realization by a stealthy Delone point
process. When the inducing lattice action is nontrivial and weakly
mixing, the Bartlett spectrum can have a nonzero continuous
component. Thus stealthy Delone point processes are abundant among
lattice-induced actions and need not have pure-point Bartlett
spectrum.

\subsection{Stealthiness and hyperuniformity}
\label{subsec:introduction-stealthiness}

Let $\eta$ be a translation-invariant simple point process on
$\mathbb R^d$ with positive intensity and local second moments.
For $f\in\mathcal S(\mathbb R^d)$, set
\[
    N_f(\omega)
    :=
    \int_{\mathbb R^d}f\,d\omega,
    \qquad
    \widehat f(\xi)
    :=
    \int_{\mathbb R^d}
    f(x)e^{-2\pi i\langle x,\xi\rangle}\,d\lambda_d(x).
\]
Here $\lambda_d$ denotes Lebesgue measure on $\mathbb R^d$. 
There is a unique positive translation-bounded Radon measure
$\sigma_\eta$ on $\mathbb R^d$ such that
\[
    \operatorname{Var}_\eta(N_f)
    =
    \int_{\mathbb R^d}
    |\widehat f(\xi)|^2\,d\sigma_\eta(\xi),
    \qquad
    f\in\mathcal S(\mathbb R^d).
\]
We call $\sigma_\eta$ the \emph{Bartlett spectrum} of $\eta$. The
process is \emph{stealthy} if $\sigma_\eta$ vanishes on an open
neighborhood of the origin.

A translation-invariant point process is \emph{hyperuniform} if
\[
    \operatorname{Var}_\eta
    \bigl(
        \omega(B_R(0))
    \bigr)
    =
    o(R^d)
    \qquad
    (R\longrightarrow\infty).
\]
Equivalently,
\[
    \sigma_\eta(B_\varepsilon(0))
    =
    o(\varepsilon^d)
    \qquad
    (\varepsilon\longrightarrow0);
\]
see \cite{TS03, Tor18, BB25}. Hence every stealthy point process is
hyperuniform.

Stealthiness also has strong consequences which are not shared by
general hyperuniform processes. Ghosh and Lebowitz proved that a
stealthy point process is maximally rigid: the configuration
outside a bounded region determines the configuration inside it.
They also proved that its realizations have uniformly bounded
holes \cite{GL18}. Further uniqueness and periodicity results have
recently been obtained by Lachi\`eze-Rey \cite{LR26}.

\subsection{Rigorous constructions and the spectral question}
\label{subsec:introduction-known-constructions}

A substantial numerical literature constructs finite stealthy
configurations by the collective-coordinate method. One places
finitely many particles in a periodic box and forces the
collective density variables to vanish at reciprocal vectors below
a prescribed cutoff. These computations produce large
configurations with little visible crystalline order
\cite{UST04,UTS06,TZS15,ZST15a,ZST15b,MST24}.

They do not by themselves produce a translation-invariant point
process on $\mathbb R^d$. Such a process would require an
infinite-volume limit, together with proofs that the spectral gap
survives and that the limiting process is not supported on
periodic configurations. The existence of disordered stealthy
point processes in dimension $d\geq2$ was explicitly posed as an
open problem in \cite[Question~5.2]{LR25}. 

Recent transport constructions give isotropic noncrystalline
point processes with arbitrarily high finite-order suppression
of the Bartlett spectrum near the origin \cite{LK26}. This
remains distinct from an exact open spectral gap.

The examples constructed below need not be periodic and may have
continuous Bartlett components, but their translation actions are
lattice-induced and therefore not weakly mixing. Thus, for
$d\geq2$, the existence of a stealthy point process with weakly
mixing translation action remains open. In dimension
one, Corollary~\ref{cor:stealthy-point-process-lattice-induced}
rules this out.

The rigorous Euclidean constructions with exact open spectral gaps
known before the present work had a different character. Randomly
translated lattices are the basic examples. Kurasov and Sarnak
constructed nonperiodic unit-mass crystalline measures in
dimension one \cite{KS20}, and Alon, Kummer, Kurasov and Vinzant
constructed higher-dimensional unit-mass Fourier quasicrystals
\cite{AKKV25}. Their natural translation hulls give nonperiodic
stealthy point processes, but their Bartlett spectra are pure
point.

A contrasting rigidity statement holds on $\mathbb Z$. The
theorem of Borichev, Sodin and Weiss implies that every
translation-invariant simple point process on $\mathbb Z$ whose
Bartlett spectrum has a nonempty open gap is periodic
\cite{BSW18}.

Our results show that the pure-point character of the known
examples is not forced over $\mathbb R^d$. For every
probability-preserving action of a full-rank lattice, we construct
a stealthy Delone point process whose translation action is the
corresponding induced $\mathbb R^d$-action. For weakly mixing
bases, the Bartlett spectrum can have a nonzero continuous
component; it may be absolutely continuous or
singular-continuous. To our knowledge, these are the first
rigorous stealthy point processes with non-pure-point Bartlett
spectrum.


\subsection{Cross-sections as point-process coordinates}
\label{subsec:introduction-cross-sections}

By a \emph{probability-preserving Borel $\mathbb R^d$-space}
$(X,\mu)$ we mean a standard Borel space equipped with a jointly
Borel, measure-preserving action
\[
    \mathbb R^d\times X\longrightarrow X,
    \qquad
    (t,x)\longmapsto t.x.
\]

Let $\mathcal N_s(\mathbb R^d)$ denote the space of locally finite
simple counting measures
\[
    \delta_P
    :=
    \sum_{p\in P}\delta_p,
    \qquad
    P\subset\mathbb R^d\ \text{locally finite},
\]
equipped with the Borel structure induced by the vague topology on
the space of positive Radon measures. We use the translation
action
\[
    t.\delta_P:=\delta_{P-t}.
\]
A \emph{point process on $\mathbb R^d$} is a Borel probability
measure on $\mathcal N_s(\mathbb R^d)$, and it is
translation-invariant if it is invariant under this action.

A Borel set $Y\subset X$ is a \emph{cross-section} if every orbit
meets $Y$ and the return-time set
\[
    Y_x
    :=
    \{t\in\mathbb R^d:t.x\in Y\}
\]
is locally finite for every $x\in X$. Since
\[
    Y_{t.x}=Y_x-t,
\]
the return-time map
\[
    \kappa_Y:X\longrightarrow\mathcal N_s(\mathbb R^d),
    \qquad
    \kappa_Y(x):=\delta_{Y_x},
\]
is equivariant:
\[
    \kappa_Y(t.x)=t.\kappa_Y(x).
\]
It is also Borel. Indeed, the incidence set
\[
    \{(x,t)\in X\times\mathbb R^d:t.x\in Y\}
\]
is Borel and has countable sections, so the claim follows from the
Lusin--Novikov theorem. Consequently,
\[
    \eta_{Y,\mu}:=(\kappa_Y)_*\mu
\]
is a translation-invariant point process, called the
\emph{return-time process} of $(X,\mu,Y)$.

Conversely, every translation-invariant simple point process
supported on nonempty configurations is the return-time process
of the cross-section consisting of configurations which contain
the origin. 

We call a cross-section \emph{generating} if there is an invariant
conull Borel set $X_0\subset X$ on which $\kappa_Y$ is injective.
By the Lusin--Souslin theorem, $\kappa_Y(X_0)$ is Borel and the
inverse of
\[
    \kappa_Y:X_0\longrightarrow\kappa_Y(X_0)
\]
is Borel. Thus $\kappa_Y$ is a measurable isomorphism between
$(X,\mu)$ and the translation action of its return-time process.

There is a useful analogy with the almost one-to-one models of
Furstenberg and Weiss \cite{FW89}. They showed that an ergodic
system with a prescribed compact rotation factor may be realized
by an invariant measure on a minimal almost one-to-one extension
of that factor. Similarly, the presence of the compact factor
$\mathbb R^d/\Gamma$ places little restriction on the
measure-theoretic complexity of a lattice-induced action: the
dynamics above this factor can be encoded by a generating
stealthy Delone cross-section. The analogy is only conceptual,
since we construct a measurable point-process model rather than
an almost one-to-one topological extension.

What does depend on the chosen cross-section is the subspace
generated by first-order linear statistics. Under the return-time
isomorphism, the centered linear statistics become
\[
    \bigl(N_f\circ\kappa_Y\bigr)^\circ,
    \qquad
    f\in\mathcal S(\mathbb R^d),
\]
where
\[
    F^\circ
    :=
    F-\int_XF\,d\mu.
\]
Changing the generating cross-section may therefore change the
Bartlett spectrum, although the full Koopman representation is
unchanged.

Osada's isomorphism theorem illustrates this distinction
\cite{Osa21}. Translation-invariant determinantal point processes
with translation-invariant kernels are measurably isomorphic, as
probability-preserving $\mathbb R^d$-actions, to homogeneous
Poisson point processes. Their canonical configuration-space
models nevertheless have different first-order linear statistics.

Starting from a lattice-induced action, we construct a generating
cross-section whose return-time configurations are uniformly
Delone and stealthy. It can moreover be chosen so that
first-order linear statistics detect a nonzero part of the
spectral type of the inducing action.


\subsection{Correlation realizability}

A classical problem asks which prescribed correlation measures
arise from a point process; see
\cite{Len73,Len75a,Len75b,CL04,KLS07,KLS11}. A stealthy point
process realizes second-order data whose Fourier transform
vanishes near the origin. Rather than prescribe the full pair
correlation, we realize this low-frequency condition by
uniformly Delone processes whose translation actions are
prescribed lattice-induced actions, and whose Bartlett spectra
may additionally detect part of the inducing spectral type.


\subsection{Lattice induction and the main results}
\label{subsec:introduction-main-results}

Let $\Gamma<\mathbb R^d$ be a full-rank lattice and let $(Z,\nu)$
be a probability-preserving Borel $\Gamma$-space. The \emph{induced
$\mathbb R^d$-space} is
\[
    \operatorname{Ind}_{\Gamma}^{\mathbb R^d}(Z)
    :=
    (\mathbb R^d\times Z)/\Gamma,
\]
where
\[
    \gamma.(t,z)
    :=
    (t+\gamma,(-\gamma).z).
\]
We write $[t,z]$ for the class of $(t,z)$. The induced action is
\[
    u.[t,z]=[t+u,z],
    \qquad
    u\in\mathbb R^d,
\]
and the space carries its standard induced probability measure. A
probability-preserving Borel $\mathbb R^d$-space is
\emph{lattice-induced} if it is measurably isomorphic to an
induced space of this form.

The map
\[
    [t,z]\longmapsto t+\Gamma
\]
gives every lattice-induced action a torus factor
$\mathbb R^d/\Gamma$. In dimension one, an ergodic
probability-preserving $\mathbb R$-action is lattice-induced if and
only if it has a nonzero eigenvalue.

The following theorem is one of the main results of the paper.

\begin{theorem}[Stealthy cross-sections and lattice induction]
\label{thm:introduction-lattice-induction}
The following assertions hold.

\begin{enumerate}
\item
For every $d\geq 1$, each lattice-induced
probability-preserving Borel $\mathbb R^d$-space admits a
generating stealthy Delone cross-section.

\item
Let $\eta$ be an ergodic translation-invariant point process on
$\mathbb R$ with positive intensity and local second moments. If
\[
    \int_{0<|\xi|<1}
    \frac{1}{\xi^2}\,d\sigma_\eta(\xi)
    <
    \infty,
\]
then the translation action on its configuration space has a
nonzero eigenvalue and is lattice-induced.
\end{enumerate}

Consequently, an ergodic probability-preserving Borel
$\mathbb R$-space is lattice-induced if and only if it admits a
generating stealthy Delone cross-section.
\end{theorem}

In the converse implication, neither generation nor the Delone
property is needed: the existence of a stealthy cross-section
system already implies lattice induction. The inverse-square
condition in part~{\rm(2)} is strictly weaker than stealthiness.

Part~{\rm(1)} is
Theorem~\ref{thm:lattice-induced-realization}. The
one-dimensional family is constructed in
Section~\ref{sec:one-dimensional-construction}; the
higher-dimensional construction and the passage to arbitrary
lattices are given in
Section~\ref{sec:layered-construction}. Part~{\rm(2)} is
Theorem~\ref{thm:spectral-lattice-induction}, and the
characterization is
Corollary~\ref{cor:one-dimensional-characterization}.

The analytic core of part~{\rm(1)} is one-dimensional. Given a
Borel automorphism $T$ of a standard Borel space $Z$, we choose a
Borel injection
\[
    b:Z\longrightarrow[-1,1]
\]
and a band-limited entire interpolation function $\psi$ satisfying
\[
    \psi(n)=
    \begin{cases}
        1,&n=0,\\
        0,&n\in\mathbb Z\setminus\{0\}.
    \end{cases}
\]
The orbit of $z$ is then encoded in
\[
    H_z(w)
    :=
    \sum_{n\in\mathbb Z}
    b(T^nz)\psi(w-n).
\]
We perturb the $1$-periodic entire function
\[
    S(w):=\sin(4\pi w)+\sin(2\pi w)
\]
and define
\[
    P_z
    :=
    \bigl\{
        x\in\mathbb R:
        S(x)+\varepsilon H_z(x)=0
    \bigr\}.
\]
The two frequencies in $S$ have different roles. The higher
frequency determines the density and growth of the zero set,
whereas the lower frequency serves as a marker which excludes
nonintegral translations. A single sine would leave its zero set
invariant under a smaller nonintegral period. The marker is
therefore what allows the encoded orbit to be recovered up to an
integral translation, and hence what ultimately gives generation.

Uniform control of the zero set is proved in
Subsection~\ref{subsec:real-zero-sets}, the Fourier
identity in Subsection~\ref{subsec:Fourier-identity}, and orbit
recovery in Subsection~\ref{subsec:orbit-recovery}. The
higher-dimensional layering argument is carried out in
Section~\ref{sec:layered-construction}.

For sufficiently small $\varepsilon$, the sets $P_z$ are uniformly
Delone. Band limitation gives the Fourier identity near the
origin, while comparison of the two frequencies gives the
orbit-separation property. In higher dimensions, these
one-dimensional sets are placed on parallel lattice layers and
then transported by a linear map to an arbitrary full-rank
lattice.

If $q:\mathbb R^d\times Z\to
\smash{\operatorname{Ind}_{\Gamma}^{\mathbb R^d}(Z)}$ denotes the
quotient map, then $q(\{0\}\times Z)$ is a cross-section.
Its return-time process is a randomly translated copy of
$\Gamma$, and its Bartlett spectrum is pure point. Its return-time
sets do not depend on $z$, so this cross-section is generating only
when $(Z,\nu)$ is essentially a one-point probability space.

The cross-section in
Theorem~\ref{thm:introduction-lattice-induction} can be chosen so
that the Bartlett spectrum also detects part of the inducing
spectral type.

A generating cross-section retains the base dynamics in the full
configuration-space action, but this alone does not imply that the
base spectrum is visible to first-order linear statistics. The
next theorem shows that the cross-section can be chosen so that it
is.

Let
\[
    \Gamma^*
    :=
    \bigl\{
        \xi\in\mathbb R^d:
        \langle\xi,\gamma\rangle\in\mathbb Z
        \text{ for every }\gamma\in\Gamma
    \bigr\}.
\]
Let $\smash{\mathfrak m_{Z,\nu}^0}$ denote the maximal spectral type of the
$\Gamma$-action on $\smash{L^2_0(Z,\nu)}$, and let
$\widetilde{\mathfrak m}_{Z,\nu}^0$ be its
$\Gamma^*$-periodic lift to $\mathbb R^d$. We write
$\mu_{\Gamma,Z}$ for the induced probability measure on
$\smash{\operatorname{Ind}_{\Gamma}^{\mathbb R^d}(Z)}$.

\begin{theorem}[Continuous Bartlett components]
\label{thm:introduction-continuous-spectrum}
Let $\Gamma\leq\mathbb R^d$ be a full-rank lattice and let
$(Z,\nu)$ be a nontrivial weakly mixing
probability-preserving Borel $\Gamma$-space. Then $\smash{\operatorname{Ind}_{\Gamma}^{\mathbb R^d}(Z)}$
admits a generating stealthy Delone cross-section $Y$. Writing
\[
    \sigma_Y
    :=
    \sigma_{Y,\mu_{\Gamma,Z}}
    =
    \sigma_{(\kappa_Y)_*\mu_{\Gamma,Z}},
\]
one has a decomposition
\[
    \sigma_Y
    =
    \sigma_Y^{\mathrm{tor}}
    +
    \sigma_Y^Z,
\]
where
\[
    \operatorname{supp}\sigma_Y^{\mathrm{tor}}
    \subset
    \Gamma^*\setminus\{0\},
    \qquad
    0\neq\sigma_Y^Z
    \ll
    \widetilde{\mathfrak m}_{Z,\nu}^0.
\]
In particular, $\sigma_Y$ has a nonzero continuous component.
\end{theorem}

This is
Theorem~\ref{thm:weakly-mixing-Bartlett-component}, proved in
Section~\ref{sec:continuous-spectral-components}. Bernoulli bases
give an absolutely continuous component, while weakly mixing bases
of singular maximal spectral type give a singular-continuous
component.

\begin{corollary}[Stealthy point processes with mixed spectra]
\label{cor:introduction-non-pure-point}
For every $d\geq1$, there exist ergodic translation-invariant
stealthy point processes on $\mathbb R^d$, supported on uniformly
Delone families of configurations, whose Bartlett spectra have,
respectively,
\begin{enumerate}
\item
a nonzero absolutely continuous component;

\item
a nonzero singular-continuous component.
\end{enumerate}
\end{corollary}


\subsection{Fourier quasicrystals and cross-sections}
\label{subsec:introduction-functional-tilings}

The Fourier quasicrystals of Kurasov and Sarnak \cite{KS20}
and Alon, Kummer, Kurasov and Vinzant \cite{AKKV25} give
stealthy point processes with pure-point Bartlett spectrum.
They also arise as return-time sets for lattice-induced torus
rotations, although from cross-sections different from those
constructed here.

For the higher-dimensional construction, let
$X\subset(\mathbb P^1)^n$ be a Lee--Yang variety, let
$L\in\mathbb R^{n\times d}$, and set
\[
    a(t):=\exp(2\pi iLt),
    \qquad
    H:=a(\mathbb R^d)\subset\mathbb T^n.
\]
We regard $\mathbb T^n$ as the standard real torus inside
$(\mathbb C^\times)^n\subset(\mathbb P^1)^n$. Translation by
$a(t)$ defines an $\mathbb R^d$-action on $H$:
\[
    t.z:=a(t)z.
\]
Equipped with Haar measure, this action is ergodic. Define
\[
    Y_X:=X\cap H.
\]
Then $Y_X$ is a cross-section for the action on $H$. Its
return-time set at $z\in H$ is
\[
\begin{aligned}
    (Y_X)_z
    &=
    \{t\in\mathbb R^d:t.z\in Y_X\}\\
    &=
    \{t\in\mathbb R^d:a(t)z\in X\}.
\end{aligned}
\]
In particular, at the identity element $1\in H$,
\[
    (Y_X)_1
    =
    \{t\in\mathbb R^d:a(t)\in X\}
    =
    \Lambda(X,L).
\]
The discreteness of the Fourier transform, with the atom at the
origin isolated, gives stealthiness, while the pure-point spectrum
of the torus rotation gives pure-point Bartlett spectrum.

Choose $d$ rows $I$ such that $L_I$ is invertible, and set
\[
    \Gamma:=L_I^{-1}\mathbb Z^d.
\]
Projection onto the coordinates in $I$, followed by the
identification
\[
    \mathbb T^d\cong\mathbb R^d/\Gamma,
    \qquad
    \exp(2\pi iL_It)\longleftrightarrow t+\Gamma,
\]
gives an equivariant factor map
\[
    H\longrightarrow\mathbb R^d/\Gamma.
\]
The rotation on $H$ is therefore induced from the
$\Gamma$-action on the fiber over $0+\Gamma$. In these examples,
the cross-section is imposed by the Lee--Yang variety. Our
construction instead chooses a coding cross-section designed to
generate the system and to make part of the inducing spectral type
visible in the Bartlett spectrum.

A separate question is whether a given nonperiodic functional
tiling can be incorporated into any translation-invariant point
process. Kolountzakis and Lev \cite{KL16}, based on an earlier
construction of Kargaev \cite{Kar83}, obtained one with translation
set
\[
    \Lambda=\{n+\alpha(n):n\in\mathbb Z\},
    \qquad
    \alpha(n)\longrightarrow0
    \quad(n\longrightarrow\pm\infty).
\]
Since $\Lambda-n\to\mathbb Z$ locally as $n\to\pm\infty$,
$\Lambda$ is not recurrent under integer translation. Poincar\'e
recurrence therefore rules out such a translation-invariant
realization. Thus the stealthy point processes furnished by
Theorem~\ref{thm:introduction-lattice-induction} have recurrent
almost-everywhere realizations, drastically unlike the
asymptotically lattice-like, nonrecurrent translation sets of
Kolountzakis and Lev.


\subsection{Density and positive random measures}
\label{subsec:introduction-further-results}

The constructions above naturally lead to a quantitative
question: how large can a stealthy gap be relative to the
intensity? Here and below, $\lambda_d$ denotes Lebesgue measure
on $\mathbb R^d$. For an open neighborhood $\Omega$ of the
origin, set
\[
    \mathcal D(\Omega)
    :=
    \sup\left\{
        \lambda_d(K):
        K\subset\mathbb R^d\ \text{compact},\
        K-K\subset\Omega
    \right\}.
\]
If $\eta$ is a translation-invariant point process, its 
\emph{intensity} $\rho_\eta$ is determined by
\[
    \int_{\mathcal N_s(\mathbb R^d)}
    \omega(f)\,d\eta(\omega)
    =
    \rho_\eta\int_{\mathbb R^d}f\,d\lambda_d,
    \qquad
    f\in C_c(\mathbb R^d).
\]
\begin{theorem}[Density bound]
\label{thm:introduction-density-bound}
Let $\eta$ be a translation-invariant point process of intensity
$\rho$, supported on uniformly separated configurations. If
\[
    \sigma_\eta(\Omega)=0,
\]
then
\[
    \rho\geq\mathcal D(\Omega).
\]
In particular, if $C\subset\Omega$ is symmetric and convex, then
\[
    \rho\geq2^{-d}\lambda_d(C).
\]
\end{theorem}

This is Theorem~\ref{thm:density-bound}, proved in
Section~\ref{sec:density-bounds}. The proof turns the spectral gap
into an exact sampling identity and applies Landau's necessary
density theorem. The bound is sharp for lattice processes.

The situation changes completely for positive random measures.
Let $\mathcal M_+(\mathbb R^d)$ denote the space of locally finite
positive Radon measures on $\mathbb R^d$.

\begin{theorem}[Positive random measures]
\label{thm:introduction-positive-measures}
Every essentially free probability-preserving Borel
$\mathbb R^d$-space admits a stealthy factor supported on measures
of the form
\[
    g\,\lambda_d,
\]
where $g$ is smooth, bounded, and strictly positive. The Bartlett
spectrum of the factor may be made to vanish on $B_R(0)$ for any
prescribed $R>0$.
\end{theorem}

This follows from
Theorem~\ref{thm:positive-random-measure-realization}. Indeed,
every essentially free Borel $\mathbb R^d$-action admits a
complete lacunary, and hence separated, Borel cross-section
\cite{Kec92}. Applying the theorem to such a cross-section gives
stealthy positive random measure factors with arbitrarily large
prescribed gaps. When a generating separated cross-section is
available, the construction gives an isomorphic realization.


\subsection{Stealthy point processes over other locally compact
abelian groups}
\label{subsec:introduction-other-groups}

Stealthiness extends naturally to locally compact abelian groups.
Let $G$ be a second-countable locally compact abelian group and
let $\widehat G$ be its dual. For a translation-invariant point
process $\eta$ on $G$, the covariance of its linear statistics is
represented by a positive measure $\sigma_\eta$ on $\widehat G$:
\[
    \operatorname{Var}_\eta(N_f)
    =
    \int_{\widehat G}
    |\widehat f(\chi)|^2\,d\sigma_\eta(\chi).
\]
The process is stealthy if $\sigma_\eta$ vanishes on a
neighborhood of the trivial character.

For $G=\mathbb Z$, Borichev, Sodin and Weiss
\cite[Theorem~3]{BSW18} imply the following.

\begin{theorem}
\label{thm:introduction-BSW}
Every translation-invariant stealthy simple point process on
$\mathbb Z$ is periodic: there is an integer $N\geq1$ such that
\[
    \omega(n+N)=\omega(n)
    \qquad(n\in\mathbb Z)
\]
for $\eta$-almost every configuration $\omega$.
\end{theorem}

A general existence result is available for groups with compact
open subgroups.

\begin{theorem}[Compact open subgroups]
\label{thm:introduction-compact-open}
Let $G$ be a second-countable locally compact abelian group and let
$(X,\mu)$ be an essentially free probability-preserving Borel
$G$-space.
\begin{enumerate}
\item If $G$ has a compact open subgroup, then $(X,\mu)$ admits a
stealthy cross-section.
\item If $G=\mathbb Q_p$, then $(X,\mu)$ admits a generating
stealthy cross-section.
\end{enumerate}
\end{theorem}

This is Theorem~\ref{thm:compact-open-cross-sections}, proved in
Section~\ref{sec:compact-open-groups}. The first assertion does
not require the cross-section to be generating and, for a discrete
group, includes the trivial choice $Y=X$. The second assertion is
therefore much stronger: every essentially free
probability-preserving $\mathbb Q_p$-action admits a generating
stealthy point-process presentation. Particular $p$-adic
cut-and-project examples were constructed earlier by the author
and Hartnick \cite{BH24}.

Since the cross-section in part~(2) is generating, its return-time
process is measurably isomorphic to the original
$\mathbb Q_p$-action. Consequently, every essentially free mixing
probability-preserving $\mathbb Q_p$-action admits a mixing
stealthy point-process realization. For example, starting from the 
translation action of a
homogeneous Poisson point process on $\mathbb Q_p$ produces a
mixing stealthy point process measurably isomorphic to that
action.

The abelian hypothesis is essential to the formulation used here,
in which the Bartlett spectrum is a measure on an ordinary dual
group. A related theory is available for homogeneous spaces
associated with Gelfand pairs, where the spherical transform
replaces the Fourier transform; see \cite{BB25,BB26} for
commutative spaces and real hyperbolic spaces. We do not pursue
that nonabelian extension in the present paper.

The known conclusions depend strongly on the acting group. Over
$\mathbb Z$, a neighborhood gap forces periodicity. Over
$\mathbb R$, the condition
\[
    \int_{0<|\xi|<1}
    \frac{1}{\xi^2}\,d\sigma_\eta(\xi)<\infty
\]
forces lattice induction, while every lattice-induced action
admits a generating stealthy Delone cross-section. Over
$\mathbb Q_p$, every essentially free action admits a generating
stealthy cross-section. For $\mathbb Z^d$, $d\geq2$, the
classification for a neighborhood gap remains open; see
\cite{LR26} for partial results under stronger spectral
assumptions.


\subsection{Proof ideas and organization}
\label{subsec:introduction-proof-ideas}

The construction begins in dimension one. For a Borel
automorphism $T$ of a standard Borel space $Z$, the orbit
\[
    (T^nz)_{n\in\mathbb Z}
\]
is encoded in the values at the integers of a band-limited entire
function. A small perturbation of a fixed periodic entire function
then has a uniformly Delone real zero set. Rouch\'e's theorem
controls the zeros, while Hadamard factorization recovers the
encoded function, and hence the orbit, from the zero set. A contour
calculation involving the logarithmic derivative gives the exact
Fourier identity near the origin. Higher dimensions are obtained
by placing the one-dimensional configurations on parallel lattice
layers and applying Poisson summation in the transverse variables.

To detect continuous Bartlett spectrum, the size of the
perturbation is retained as a parameter. Differentiating the zero
locations at the unperturbed configuration produces a linear
statistic with a nonzero projection onto the spectral subspace of
the inducing action. The one-dimensional converse begins by
showing that a centered interval count is an $L^2$-coboundary for
a suitable time map. Since the count is integer-valued,
exponentiating the transfer function produces an eigenfunction,
from which lattice induction follows. The density bound uses an
exact sampling identity and Landau's theorem, the positive-random-
measure construction uses high-pass convolution, and the
$\mathbb Q_p$ construction uses finite cyclic coding inside
compact-open orbits.

Sections~\ref{sec:cross-sections} and
\ref{sec:lattice-induced-actions} develop the cross-section,
spectral and lattice-induction formalism.
Section~\ref{sec:spectral-lattice-induction} proves the
one-dimensional converse.
Sections~\ref{sec:one-dimensional-construction} and
\ref{sec:layered-construction} give the one- and
higher-dimensional constructions, and
Section~\ref{sec:continuous-spectral-components} detects the
inducing spectrum in the Bartlett spectrum.
Sections~\ref{sec:density-bounds},
\ref{sec:positive-random-measures} and
\ref{sec:compact-open-groups} prove the remaining results.


\subsection{Acknowledgments}

The author was supported by the Swedish Research Council under
grant VR~11253322. He is deeply grateful to Mattias Byl\'ehn,
Alexander Fish, Tobias Hartnick, Pavel Kurasov, G\"unter Last, Luca Lotz, and Michael Klatt for
valuable discussions related to this work.

\subsection*{Use of AI tools}

During the preparation of this manuscript, the author used
ChatGPT (OpenAI, GPT-5.6 Sol) for exploratory mathematical
discussion, organization, and drafting assistance, and Aristotle
(Harmonic) as an additional informal check on selected arguments.
All AI-assisted material was independently verified and edited by
the author, who takes full responsibility for the mathematical
content and the final text.



\section{Cross-section systems and spectra}
\label{sec:cross-sections}

Throughout, $\langle\cdot,\cdot\rangle$ and $\|\cdot\|$ denote the
Euclidean inner product and norm on $\mathbb R^d$, and
\[
    B_r(v):=\{u\in\mathbb R^d:\|u-v\|<r\}.
\]
For a Borel action of $\mathbb R^d$ on a Borel space $X$, we write
$t.x$ for the image of $x$ under $t\in\mathbb R^d$.

\subsection{Point configurations and return times}
\label{subsec:return-times}

Let $\mathcal M_+(\mathbb R^d)$ denote the space of positive Radon
measures on $\mathbb R^d$, equipped with the vague topology
generated by the maps
\[
    \omega\longmapsto\omega(f)
    :=
    \int_{\mathbb R^d}f\,d\omega,
    \qquad
    f\in C_c(\mathbb R^d).
\]

We identify a locally finite set $P\subset\mathbb R^d$ with its
counting measure
\[
    \delta_P:=\sum_{p\in P}\delta_p,
\]
and write
\[
    \mathcal N_s(\mathbb R^d)
    :=
    \left\{
        \delta_P:
        P\cap K\text{ is finite for every compact }
        K\subset\mathbb R^d
    \right\}.
\]
This space carries the Borel structure induced from
$\mathcal M_+(\mathbb R^d)$. A point process on $\mathbb R^d$ is
a Borel probability measure on $\mathcal N_s(\mathbb R^d)$.

The translation action is
\[
    (t.\omega)(f)
    :=
    \omega\bigl(f(\,\cdot-t)\bigr),
\]
so that $t.\delta_P=\delta_{P-t}$. A point process $\eta$ is
\emph{translation-invariant} if $t_*\eta=\eta$ for every
$t\in\mathbb R^d$, and \emph{ergodic} if every invariant Borel set
has measure zero or one. It has \emph{local second moments} if
\[
    \int_{\mathcal N_s(\mathbb R^d)}
    \omega(K)^2\,d\eta(\omega)<\infty
    \qquad
    \text{for every compact }K\subset\mathbb R^d.
\]

A set $P\subset\mathbb R^d$ is \emph{$r$-separated} if
\[
    \inf_{\substack{p,q\in P\\p\neq q}}\|p-q\|\geq r,
\]
and \emph{$R$-relatively dense} if
\[
    \sup_{x\in\mathbb R^d}\operatorname{dist}(x,P)\leq R.
\]
It is \emph{Delone} if both conditions hold for some $r,R>0$.
A family is \emph{uniformly Delone} if the same constants work
for every member.

Let $X$ be a standard Borel space equipped with a Borel
$\mathbb R^d$-action. For a Borel set $Y\subset X$, define
\[
    Y_x:=\{t\in\mathbb R^d:t.x\in Y\}.
\]

\begin{definition}
\label{def:cross-section}
A Borel set $Y\subset X$ is a \emph{cross-section} if every orbit
meets $Y$ and every return-time set $Y_x$ is locally finite.
\end{definition}

Since
\[
    Y_{t.x}=Y_x-t,
\]
the return-time map
\[
    \kappa_Y:X\longrightarrow\mathcal N_s(\mathbb R^d),
    \qquad
    \kappa_Y(x):=\delta_{Y_x},
\]
is equivariant. It is Borel by the Lusin--Novikov theorem, applied
to the incidence set
\[
    \{(x,t)\in X\times\mathbb R^d:t.x\in Y\}.
\]

A cross-section is \emph{separated} if, for some $r>0$,
\[
    Y_y\cap B_r(0)=\{0\}
    \qquad
    (y\in Y),
\]
or equivalently if all return-time sets are uniformly
$r$-separated. It is \emph{cocompact} if $K.Y=X$ for some compact
$K\subset\mathbb R^d$, equivalently if its return-time sets are
uniformly relatively dense. A cross-section which is both
separated and cocompact is called a \emph{Delone cross-section}.
For the general theory of separated cross-sections and transverse
measures, see \cite[Sections~3 and~4]{BHK25}.

Let $\mu$ be a translation-invariant Borel probability measure on
$X$. The associated \emph{return-time process} is
\[
    \eta_{Y,\mu}:=(\kappa_Y)_*\mu.
\]
It is translation-invariant, and is ergodic whenever $\mu$ is
ergodic.

\begin{definition}
\label{def:cross-section-system}
A \emph{cross-section system} is a triple $(X,\mu,Y)$ for which
the return-time process has local second moments, equivalently
\[
    \int_X\#(Y_x\cap K)^2\,d\mu(x)<\infty
    \qquad
    \text{for every compact }K\subset\mathbb R^d.
\]
\end{definition}

When the system is fixed, we usually write $\eta_Y$ in place of
$\eta_{Y,\mu}$.

All equivariant maps and isomorphisms below are understood after
restriction to invariant conull Borel sets. A measurable
$\mathbb R^d$-isomorphism is an equivariant, measure-preserving
Borel bijection between such sets. The same convention applies to
cross-sections and to pointwise representatives of objects
initially defined only almost everywhere.

\begin{definition}
\label{def:generating-cross-section}
A cross-section system $(X,\mu,Y)$ is \emph{generating} if
$\kappa_Y$ is injective on an invariant conull Borel subset of
$X$.
\end{definition}

By the Lusin--Souslin theorem, a generating return-time map is a
measurable $\mathbb R^d$-isomorphism onto its return-time
point-process system.

\subsection{Linear statistics and the Bartlett spectrum}
\label{subsec:spectral-measures}

For $f\in L^1(\lambda_d)$, we use the Fourier transform
\[
    \widehat f(\xi)
    :=
    \int_{\mathbb R^d}
    f(x)e^{-2\pi i\langle x,\xi\rangle}\,d\lambda_d(x),
    \qquad
    \xi\in\mathbb R^d.
\]

Let $\eta$ be a translation-invariant point process with local
second moments. Its first moment measure is translation-invariant,
and hence there is a unique $\rho_\eta\geq0$ such that
\[
    \int_{\mathcal N_s(\mathbb R^d)}
    \omega(f)\,d\eta(\omega)
    =
    \rho_\eta\int_{\mathbb R^d}f\,d\lambda_d
    \qquad
    (f\in C_c(\mathbb R^d)).
\]
The number $\rho_\eta$ is called the \emph{intensity} of $\eta$.

For $f\in\mathcal S(\mathbb R^d)$, write
\[
    N_f(\omega)
    :=
    \omega(f)
    =
    \sum_{p\in\operatorname{supp}\omega}f(p).
\]
Under the local second-moment assumption, this series converges
absolutely almost surely, defines an element of $L^2(\eta)$, and
satisfies
\[
    \int_{\mathcal N_s(\mathbb R^d)}
    N_f\,d\eta
    =
    \rho_\eta\int_{\mathbb R^d}f\,d\lambda_d.
\]
Set
\[
    N_f^\circ
    :=
    N_f-\rho_\eta\int_{\mathbb R^d}f\,d\lambda_d.
\]

Proposition~8.2.I and Definition~8.2.II of \cite{DVJ03},
after rescaling the spectral variable to match our Fourier
convention, give a unique positive Radon measure $\sigma_\eta$
on $\mathbb R^d$ such that
\[
    \int_{\mathcal N_s(\mathbb R^d)}
    N_f^\circ\overline{N_g^\circ}\,d\eta
    =
    \int_{\mathbb R^d}
    \widehat f(\xi)\overline{\widehat g(\xi)}
    \,d\sigma_\eta(\xi)
\]
for all $f,g\in\mathcal S(\mathbb R^d)$. The cited result states
the corresponding variance identity; the displayed identity
follows by polarization.

\begin{definition}
\label{def:Bartlett-spectrum}
The measure $\sigma_\eta$ is called the \emph{Bartlett spectral
measure}, or the \emph{Bartlett spectrum}, of $\eta$.
\end{definition}

The measure $\sigma_\eta$ is translation bounded
\cite[Proposition~8.2.II(iv)]{DVJ03}; consequently, when $d=1$,
\[
    \int_{|\xi|\geq1}
    \frac{1}{\xi^2}\,d\sigma_\eta(\xi)<\infty.
\]

\begin{definition}
\label{def:stealthy-process}
Let $\Omega\subset\mathbb R^d$ be open. A translation-invariant
point process $\eta$ is \emph{stealthy on $\Omega$} if
\[
    \sigma_\eta(\Omega)=0.
\]
It is \emph{stealthy} if it is stealthy on some open neighborhood
of the origin.
\end{definition}

For a cross-section system $(X,\mu,Y)$, we write
\[
    \rho_{Y,\mu}:=\rho_{\eta_{Y,\mu}},
    \qquad
    \sigma_{Y,\mu}:=\sigma_{\eta_{Y,\mu}},
\]
and abbreviate these to $\rho_Y$ and $\sigma_Y$ when the system is
fixed. We call the cross-section system stealthy if its
return-time process is stealthy.

\subsection{Koopman spectral measures}
\label{subsec:koopman-spectra}

Let $(X,\mu)$ be a probability-preserving Borel
$\mathbb R^d$-space. Its Koopman representation on
$L^2(X,\mu)$ is
\[
    U_tF(x):=F((-t).x),
    \qquad
    t\in\mathbb R^d.
\]
For every $F\in L^2(X,\mu)$, there is a unique finite positive
measure $\varsigma_F$ on $\mathbb R^d$ such that
\[
    \langle U_tF,F\rangle_{L^2(X,\mu)}
    =
    \int_{\mathbb R^d}
    e^{2\pi i\langle t,\xi\rangle}\,d\varsigma_F(\xi)
    \qquad
    (t\in\mathbb R^d).
\]
We call $\varsigma_F$ the \emph{Koopman spectral measure} of
$F$. The \emph{maximal spectral type} of $(X,\mu)$ is the least
measure class dominating $\varsigma_F$ for every
$F\in L^2(X,\mu)$; we denote it by $\mathfrak m_{X,\mu}$.

A nonzero function $F\in L^2(X,\mu)$ is an
\emph{eigenfunction} with eigenvalue $\xi\in\mathbb R^d$ if
\[
    U_tF=e^{2\pi i\langle t,\xi\rangle}F
    \qquad
    (t\in\mathbb R^d),
\]
equivalently if
\[
    \varsigma_F=\|F\|_2^2\delta_\xi.
\]
By \cite[Proposition~5.7]{ABC25}, an $L^2$-eigenfunction has a
Borel representative on an invariant conull Borel set
$X_0\subset X$ satisfying
\[
    F(t.x)
    =
    e^{-2\pi i\langle t,\xi\rangle}F(x)
    \qquad
    (t\in\mathbb R^d,\ x\in X_0).
\]
We always choose eigenfunctions in this form.

For a translation-invariant point process $\eta$, the Bartlett
and Koopman spectral measures are related by
\[
    d\varsigma_{N_f^\circ}(\xi)
    =
    |\widehat f(\xi)|^2\,d\sigma_\eta(\xi),
    \qquad
    f\in\mathcal S(\mathbb R^d).
\]
For a cross-section system $(X,\mu,Y)$, set
\[
    N_f^Y:=N_f\circ\kappa_Y,
    \qquad
    (N_f^Y)^\circ
    :=
    N_f^Y-\rho_Y\int_{\mathbb R^d}f\,d\lambda_d.
\]
The equivariance of $\kappa_Y$ gives
\[
    d\varsigma_{(N_f^Y)^\circ}(\xi)
    =
    |\widehat f(\xi)|^2\,d\sigma_Y(\xi).
\]

Let $\mathcal H_\eta\subset L^2(\eta)$ be the closed
$\mathbb R^d$-invariant subspace generated by the centered
linear statistics. Since one may choose
$f\in\mathcal S(\mathbb R^d)$ with $\widehat f$ nowhere zero,
the preceding identity shows that $\sigma_\eta$ represents the
maximal spectral type on $\mathcal H_\eta$. In general,
$\mathcal H_\eta$ may be strictly smaller than
$L^2_0(\eta)$, so the Bartlett spectrum need not represent the
maximal spectral type of the full point-process action.

In the terminology of diffraction theory, $\mathcal H_\eta$ is
the \emph{diffraction subspace}, or \emph{Dworkin subspace}, of
the point-process action. For an ergodic point process, if
$\gamma_\eta$ denotes the autocorrelation in the usual
normalization, then
\[
    \widehat{\gamma_\eta}
    =
    \rho_\eta^2\delta_0+\sigma_\eta.
\]
The map $f\mapsto N_f^\circ$ is the corresponding Dworkin
embedding. In general this subspace may be proper in
$L^2_0(\eta)$, so diffraction need not detect the full
dynamical spectrum; see
\cite{Gou03,BL04,DM08,BLvE15}.


\section{Lattice-induced actions}
\label{sec:lattice-induced-actions}

This section recalls lattice induction and describes cross-sections
on induced spaces in terms of equivariant families of point sets.
Fix a full-rank lattice $\Gamma\leq\mathbb R^d$. Its dual lattice
is
\[
    \Gamma^*
    :=
    \{\xi\in\mathbb R^d:
      \langle\gamma,\xi\rangle\in\mathbb Z
      \text{ for every }\gamma\in\Gamma\}.
\]

\subsection{Induced spaces}
\label{subsec:induced-spaces}

Let $(Z,\nu)$ be a standard probability space equipped with a
measure-preserving Borel action of $\Gamma$. Consider the diagonal
$\Gamma$-action on $\mathbb R^d\times Z$ given by
\[
    \gamma.(t,z):=(t+\gamma,(-\gamma).z).
\]
The induced space is
\[
    \operatorname{Ind}_{\Gamma}^{\mathbb R^d}(Z)
    :=
    (\mathbb R^d\times Z)/\Gamma.
\]
We write $[t,z]$ for the class of $(t,z)$. Translation in the
first coordinate gives a Borel $\mathbb R^d$-action
\[
    s.[t,z]:=[s+t,z].
\]

Let $D\subset\mathbb R^d$ be a Borel fundamental domain for
$\Gamma$, and set
\[
    \operatorname{covol}(\Gamma):=\lambda_d(D).
\]
If
\[
    q:\mathbb R^d\times Z
    \longrightarrow
    \operatorname{Ind}_{\Gamma}^{\mathbb R^d}(Z)
\]
is the quotient map, the induced probability measure is
\[
    \mu_{\Gamma,Z}
    :=
    q_*\left(
        \frac{\lambda_d|_D}{\operatorname{covol}(\Gamma)}
        \otimes\nu
    \right).
\]
It is independent of the choice of $D$ and invariant under the
induced $\mathbb R^d$-action. The induced action is ergodic if and
only if the original $\Gamma$-action is ergodic.

A probability-preserving Borel $\mathbb R^d$-space is
\emph{$\Gamma$-induced} if it is measurably isomorphic to
\[
    \left(
        \operatorname{Ind}_{\Gamma}^{\mathbb R^d}(Z),
        \mu_{\Gamma,Z}
    \right)
\]
for some probability-preserving Borel $\Gamma$-space $(Z,\nu)$.
It is \emph{lattice-induced} if it is $\Gamma$-induced for some
full-rank lattice $\Gamma\leq\mathbb R^d$.

\begin{lemma}
\label{lem:independent-eigenvalues-lattice-induction}
Let $(X,\mu)$ be an ergodic probability-preserving Borel
$\mathbb R^d$-space. If its Koopman representation has $d$
linearly independent eigenvalues
\[
    \xi_1,\ldots,\xi_d\in\mathbb R^d,
\]
then $(X,\mu)$ is lattice-induced. In particular, when $d=1$, a single nonzero Koopman eigenvalue implies that the action is lattice-induced.
\end{lemma}

\begin{proof}
Let $F_1,\ldots,F_d$ be corresponding eigenfunctions. Their
moduli are invariant and hence constant almost everywhere.
After normalization, complex conjugation and modification on an
invariant null set, regard them as Borel maps
\[
    E_j:X\longrightarrow\mathbb T:=\mathbb R/\mathbb Z
\]
satisfying
\[
    E_j(t.x)
    =
    E_j(x)+\langle t,\xi_j\rangle
    \qquad
    (t\in\mathbb R^d,\ x\in X).
\]

Define
\[
    B:\mathbb R^d\longrightarrow\mathbb R^d,
    \qquad
    Bt
    :=
    \bigl(
        \langle t,\xi_1\rangle,\ldots,
        \langle t,\xi_d\rangle
    \bigr),
\]
and set
\[
    \Gamma:=B^{-1}\mathbb Z^d.
\]
The linear independence of the $\xi_j$ makes $B$ invertible, so
$\Gamma$ is a full-rank lattice.

The map
\[
    E:=(E_1,\ldots,E_d):X\longrightarrow\mathbb T^d
\]
satisfies
\[
    E(t.x)=E(x)+Bt.
\]
The action $t.u=u+Bt$ on $\mathbb T^d$ is transitive. Hence
$E_*\mu$ is Haar measure and $E$ is surjective.

Set
\[
    Z:=E^{-1}(0).
\]
Then $Z$ is Borel and $\Gamma$-invariant. The map
\[
    \Phi:
    \operatorname{Ind}_{\Gamma}^{\mathbb R^d}(Z)
    \longrightarrow X,
    \qquad
    \Phi([t,z]):=t.z,
\]
is well defined and equivariant. If $t.z=t'.z'$, then
\[
    B(t-t')\in\mathbb Z^d,
\]
so $t-t'\in\Gamma$ and
\[
    z'=(t-t').z.
\]
Thus $\Phi$ is injective, while surjectivity follows from that of
$E$. Hence $\Phi$ is a Borel isomorphism.

It remains to identify the measure. Put
\[
    D:=B^{-1}[0,1)^d
\]
and identify the induced space with $D\times Z$. For
\[
    \widetilde\mu:=\Phi^{-1}_*\mu,
\]
the marginal on $D$ is normalized Lebesgue measure, so
\[
    \widetilde\mu
    =
    \frac{1}{\operatorname{covol}(\Gamma)}
    \int_D\delta_t\otimes\nu_t\,d\lambda_d(t)
\]
for a measurable family of probability measures $\nu_t$ on $Z$.

Let $\mathcal A$ be a countable algebra generating the Borel
$\sigma$-algebra of $Z$. Invariance under translations which
remain inside $D$ gives, for $A\in\mathcal A$,
\[
    \int_C\nu_t(A)\,d\lambda_d(t)
    =
    \int_{C+u}\nu_t(A)\,d\lambda_d(t)
\]
whenever $C,C+u\subset D$. It follows that
$t\mapsto\nu_t(A)$ is constant almost everywhere. Applying this
simultaneously to $A\in\mathcal A$, we obtain a probability
measure $\nu$ on $Z$ such that
\[
    \nu_t=\nu
    \qquad
    \text{for almost every }t\in D.
\]

For $\gamma\in\Gamma$, translation by $\gamma$ acts in these
coordinates by
\[
    (t,z)\longmapsto(t,\gamma.z).
\]
The invariance of $\widetilde\mu$ therefore implies that $\nu$ is
$\Gamma$-invariant. Consequently,
\[
    \widetilde\mu=\mu_{\Gamma,Z},
\]
and $(X,\mu)$ is lattice-induced.
\end{proof}


\subsection{Equivariant families of point sets}
\label{subsec:equivariant-families}

Let $(Z,\nu)$ be a probability-preserving Borel $\Gamma$-space.
A family of locally finite sets
\[
    z\longmapsto P_z\subset\mathbb R^d
\]
is \emph{Borel} if the map $z\mapsto\delta_{P_z}$ is Borel, and
\emph{$\Gamma$-equivariant} if
\[
    P_{\gamma.z}=P_z-\gamma
    \qquad
    (\gamma\in\Gamma,\ z\in Z).
\]

Suppose that every $P_z$ is nonempty, and define
\[
    Y_P
    :=
    \bigl\{
        [t,z]\in
        \operatorname{Ind}_{\Gamma}^{\mathbb R^d}(Z):
        t\in P_z
    \bigr\}.
\]
Equivariance makes this definition independent of the chosen
representative.

\begin{proposition}
\label{prop:equivariant-families}
The set $Y_P$ is a Borel cross-section, and
\[
    (Y_P)_{[t,z]}=P_z-t,
    \qquad
    \kappa_{Y_P}([t,z])=\delta_{P_z-t}.
\]
If the family $(P_z)_{z\in Z}$ is uniformly Delone, then $Y_P$ is
a Delone cross-section and
\[
    \bigl(
        \operatorname{Ind}_{\Gamma}^{\mathbb R^d}(Z),
        \mu_{\Gamma,Z},
        Y_P
    \bigr)
\]
is a cross-section system.

Suppose, moreover, that there is a $\Gamma$-invariant conull
Borel set $Z_0\subset Z$ such that
\[
    P_{z'}=P_z-u
    \quad\Longrightarrow\quad
    u\in\Gamma
    \ \text{ and }\
    z'=u.z
\]
for all $z,z'\in Z_0$ and $u\in\mathbb R^d$. Then $Y_P$ is
generating.
\end{proposition}

\begin{proof}
The incidence set
\[
    \{(t,z)\in\mathbb R^d\times Z:t\in P_z\}
\]
is Borel and invariant under the diagonal $\Gamma$-action. Since
the quotient map has countable fibers, its image $Y_P$ is Borel.

For $s\in\mathbb R^d$,
\[
    s.[t,z]\in Y_P
    \quad\Longleftrightarrow\quad
    s+t\in P_z,
\]
which gives the formulas for the return-time set and return-time
map. Nonemptiness and local finiteness of the sets $P_z$ show that
$Y_P$ is a cross-section. If the family is uniformly Delone, then
the same formulas show that $Y_P$ is separated and cocompact.
Uniform separation also bounds
\[
    \#\bigl((P_z-t)\cap K\bigr)
\]
uniformly in $z$ and $t$, for every compact $K$, so the return-time
process has local second moments.

Finally, let
\[
    X_0:=q(\mathbb R^d\times Z_0).
\]
This is an invariant conull Borel set. If
\[
    \kappa_{Y_P}([t,z])
    =
    \kappa_{Y_P}([t',z'])
\]
for points of $X_0$, then
\[
    P_{z'}=P_z-(t-t').
\]
The orbit-separation assumption gives
\[
    t-t'\in\Gamma,
    \qquad
    z'=(t-t').z,
\]
and therefore $[t,z]=[t',z']$.
\end{proof}

\begin{proposition}[Fourier criterion for stealthiness]
\label{prop:stealthiness-criterion}
Let $(P_z)_{z\in Z}$ be a Borel $\Gamma$-equivariant uniformly
Delone family, and let $Y_P$ be the associated cross-section.
Suppose that there are $\rho>0$ and an open neighborhood
$\Omega$ of the origin such that
\[
    \widehat{\delta}_{P_z}
    =
    \rho\delta_0
    \qquad\text{on }\Omega
\]
for $\nu$-almost every $z$, in the sense of tempered
distributions. Then the associated cross-section system has intensity $\rho$.
Moreover, for every $f\in\mathcal S(\mathbb R^d)$ satisfying
\[
    \operatorname{supp}\widehat f\subset\Omega,
\]
one has
\[
    N_f^{Y_P}(x)
    =
    \rho\int_{\mathbb R^d}f\,d\lambda_d
\]
for $\mu_{\Gamma,Z}$-almost every $x$. Consequently,
\[
    \sigma_{Y_P,\mu_{\Gamma,Z}}(\Omega)=0.
\]
In particular, the cross-section system is stealthy.
\end{proposition}

\begin{proof}
Represent the induced measure on a fundamental domain. For
$\mu_{\Gamma,Z}$-almost every $x=[t,z]$,
\[
    \kappa_{Y_P}(x)=\delta_{P_z-t},
\]
and translation preserves the asserted distributional identity
on $\Omega$. Hence Fourier inversion gives
\[
    N_f\bigl(\kappa_{Y_P}(x)\bigr)
    =
    \rho\int_{\mathbb R^d}f\,d\lambda_d
\]
whenever $f\in\mathcal S(\mathbb R^d)$ and
\[
    \operatorname{supp}\widehat f\subset\Omega.
\]

Choosing such an $f$ with nonzero integral and taking expectations
shows that the intensity is $\rho$. Thus
$(N_f^{Y_P})^\circ=0$ almost everywhere, and the Bartlett identity
gives
\[
    \int_{\mathbb R^d}
    |\widehat f(\xi)|^2\,
    d\sigma_{Y_P,\mu_{\Gamma,Z}}(\xi)=0.
\]
Since $\widehat f$ may be chosen arbitrarily in
$C_c^\infty(\Omega)$, it follows that
\[
    \sigma_{Y_P,\mu_{\Gamma,Z}}(\Omega)=0.
\]
\end{proof}

Accordingly, the realization theorem reduces to constructing
$\Gamma$-equivariant uniformly Delone families satisfying the
Fourier identity and the orbit-separation condition above.



\section{A spectral criterion for lattice induction in dimension one}
\label{sec:spectral-lattice-induction}

We prove that a weighted integrability condition on the Bartlett
spectrum near the origin forces an ergodic point process on
$\mathbb R$ to be lattice-induced. We call a point process
lattice-induced if its translation action on configuration space
is lattice-induced in the sense of
Section~\ref{sec:lattice-induced-actions}.

\subsection{The point-process criterion}
\label{subsec:point-process-lattice-criterion}

We shall use the standard $L^2$-coboundary criterion. If $T$ is an
invertible probability-preserving transformation, $U$ is its
Koopman operator, and $\nu_F$ is the spectral measure of
$F\in L^2$, then
\[
    F\in\operatorname{Ran}(U-I)
    \quad\Longleftrightarrow\quad
    \int_{\mathbb T}
    \frac{1}{|z-1|^2}\,d\nu_F(z)<\infty.
\]
See, for instance, \cite[Proposition~3.1(i)]{CL21}.

\begin{theorem}[Spectral criterion for lattice induction]
\label{thm:spectral-lattice-induction}
Let $\eta$ be an ergodic translation-invariant point process on
$\mathbb R$ with positive intensity $\rho$ and local second
moments. If
\[
    \int_{0<|\xi|<1}
    \frac{1}{\xi^2}\,d\sigma_\eta(\xi)<\infty,
\]
then the translation action of $\eta$ has a nonzero eigenvalue.
Consequently, $\eta$ is lattice-induced.
\end{theorem}

\begin{proof}
The argument has three steps. We first show that, for a suitable
time-$a$ map, the centered number of points in $[0,a)$ is an
$L^2$-coboundary. Exponentiating the transfer function then gives
an eigenfunction of the time-$a$ transformation. Finally, its
spectral measure for the full flow is supported on a discrete
coset which avoids the origin, and therefore has a nonzero atom.

Since $\eta$ is ergodic, the spectral measure of every centered
linear statistic has no atom at the origin. Choosing
$f\in\mathcal S(\mathbb R)$ with $\widehat f(0)\neq0$ and using
\[
    d\varsigma_{N_f^\circ}(\xi)
    =
    |\widehat f(\xi)|^2\,d\sigma_\eta(\xi)
\]
gives
\[
    \sigma_\eta(\{0\})=0.
\]
Together with the hypothesis and the one-dimensional consequence
of translation boundedness recorded in
Subsection~\ref{subsec:spectral-measures}, this yields
\[
    \int_{\mathbb R\setminus\{0\}}
    \frac{1}{\xi^2}\,d\sigma_\eta(\xi)<\infty.
\]

Choose $a>0$ such that
\[
    \rho a\notin\mathbb Z,
\]
and let
\[
    T_a\omega:=(-a).\omega.
\]
Consider the centered interval count
\[
    A_a(\omega)
    :=
    \omega([0,a))-\rho a.
\]
By local second moments, $A_a\in L^2(\eta)$.

Let $(\chi_n)$ be a nonnegative smooth approximate identity with
\[
    \int_{\mathbb R}\chi_n\,d\lambda_1=1,
    \qquad
    \operatorname{supp}\chi_n\subset(-1/n,1/n),
\]
and set
\[
    f_n:=\mathbf 1_{[0,a)}*\chi_n.
\]
The first moment measure gives
\[
    \int\omega(\{0,a\})\,d\eta(\omega)=0.
\]
Thus, almost every configuration has no point at either endpoint.
Since $0\leq f_n\leq1$ and the relevant supports lie in a fixed
compact interval, dominated convergence and local second moments
give
\[
    N_{f_n}^\circ\longrightarrow A_a
    \qquad\text{in }L^2(\eta).
\]

Moreover,
\[
    \widehat f_n
    =
    \widehat{\mathbf 1}_{[0,a)}\,\widehat\chi_n,
    \qquad
    |\widehat\chi_n|\leq1,
    \qquad
    \widehat\chi_n(\xi)\longrightarrow1.
\]
Since
\[
    \left|
        \widehat{\mathbf 1}_{[0,a)}(\xi)
    \right|
    \leq
    \min\left\{a,\frac{1}{\pi|\xi|}\right\},
\]
the preceding weighted integrability allows us to pass to the
limit in the Bartlett identity. Hence the spectral measure of
$A_a$ for the full $\mathbb R$-action is
\[
    d\varsigma_{A_a}(\xi)
    =
    \left|
        \widehat{\mathbf 1}_{[0,a)}(\xi)
    \right|^2
    d\sigma_\eta(\xi).
\]

Let $\nu_a$ be the spectral measure of $A_a$ for the transformation
$T_a$, and define
\[
    q_a:\mathbb R\longrightarrow\mathbb T,
    \qquad
    q_a(\xi):=e^{2\pi ia\xi}.
\]
Comparison of Fourier coefficients gives
\[
    \nu_a
    =
    (q_a)_*
    \left(
        \left|
            \widehat{\mathbf 1}_{[0,a)}
        \right|^2\sigma_\eta
    \right).
\]
Since
\[
    q_a^{-1}(\{1\})=a^{-1}\mathbb Z,
\]
the Fourier transform of $\mathbf 1_{[0,a)}$ vanishes at every
nonzero point of this set, while $\sigma_\eta(\{0\})=0$.
Consequently,
\[
    \nu_a(\{1\})=0.
\]

For $\xi\notin a^{-1}\mathbb Z$,
\[
    \widehat{\mathbf 1}_{[0,a)}(\xi)
    =
    \frac{1-e^{-2\pi ia\xi}}{2\pi i\xi},
\]
and therefore
\[
    \frac{
        \left|
            \widehat{\mathbf 1}_{[0,a)}(\xi)
        \right|^2
    }{
        |e^{2\pi ia\xi}-1|^2
    }
    =
    \frac{1}{4\pi^2\xi^2}.
\]
The pushforward identity now gives
\[
\begin{aligned}
    \int_{\mathbb T}
    \frac{1}{|z-1|^2}\,d\nu_a(z)
    &=
    \frac{1}{4\pi^2}
    \int_{\mathbb R\setminus a^{-1}\mathbb Z}
    \frac{1}{\xi^2}\,d\sigma_\eta(\xi)\\
    &<\infty.
\end{aligned}
\]
By the $L^2$-coboundary criterion, there is $Q_a\in L^2(\eta)$
such that
\[
    A_a=Q_a\circ T_a-Q_a.
\]
Replacing $Q_a$ by its real part, we may assume that it is
real-valued.

Set
\[
    E_a:=e^{2\pi iQ_a}.
\]
Since $\omega([0,a))$ is integer-valued,
\[
\begin{aligned}
    E_a\circ T_a
    &=
    e^{2\pi i(Q_a+A_a)}\\
    &=
    e^{2\pi i(\omega([0,a))-\rho a)}E_a\\
    &=
    e^{-2\pi i\rho a}E_a.
\end{aligned}
\]
Thus $E_a$ is an eigenfunction of the time-$a$ transformation.

Let $\varsigma_{E_a}$ be its spectral measure for the full
$\mathbb R$-action. The preceding identity implies
\[
    \operatorname{supp}\varsigma_{E_a}
    \subset
    -\rho+\frac1a\mathbb Z.
\]
Since $\rho a\notin\mathbb Z$, this set does not contain the
origin. The nonzero finite measure $\varsigma_{E_a}$ is supported
on a countable set and therefore has an atom at some
$\lambda\neq0$. The corresponding spectral projection is a
nonzero eigenfunction of the full action. Applying
Lemma~\ref{lem:independent-eigenvalues-lattice-induction} in
dimension one shows that $\eta$ is lattice-induced.
\end{proof}

\begin{remark}
The passage from an integer-valued coboundary to an eigenfunction
used above goes back to Halász \cite{Hal76}.
\end{remark}

\begin{corollary}
\label{cor:spectral-decay-lattice-induction}
Let $\eta$ be an ergodic translation-invariant point process on
$\mathbb R$ with positive intensity and local second moments.
Suppose that there are $C>0$ and $\alpha>0$ such that
\[
    \sigma_\eta([-\varepsilon,\varepsilon])
    \leq
    C\varepsilon^{2+\alpha}
\]
for all sufficiently small $\varepsilon>0$. Then $\eta$ is
lattice-induced.
\end{corollary}

\begin{proof}
Choose $n_0$ so that the assumed estimate holds for
$\varepsilon\leq2^{-n_0}$. Decomposing into dyadic annuli gives
\[
\begin{aligned}
    \int_{0<|\xi|<2^{-n_0}}
    \frac{1}{\xi^2}\,d\sigma_\eta(\xi)
    &\leq
    \sum_{n\geq n_0}
    2^{2n+2}
    \sigma_\eta([-2^{-n},2^{-n}])\\
    &\leq
    4C\sum_{n\geq n_0}2^{-\alpha n}
    <\infty.
\end{aligned}
\]
The integral over
$2^{-n_0}\leq|\xi|<1$ is finite because $\sigma_\eta$ is Radon,
so Theorem~\ref{thm:spectral-lattice-induction} applies.
\end{proof}

\begin{remark}[Positive $\alpha$ is necessary]
\label{rem:sine-process-threshold}
The exponent $2$ is insufficient in
Corollary~\ref{cor:spectral-decay-lattice-induction}. For the
unit-intensity sine-kernel determinantal point process,
\[
    K(x)=\frac{\sin(\pi x)}{\pi x}.
\]
The determinantal covariance formula gives
\[
    d\sigma_\eta(\xi)
    =
    \left(
        1-\widehat{|K|^2}(\xi)
    \right)d\xi
    =
    \min\{|\xi|,1\}\,d\xi.
\]
Hence
\[
    \sigma_\eta([-\varepsilon,\varepsilon])
    =
    \varepsilon^2
    \qquad
    (0<\varepsilon<1).
\]
Its translation action is mixing \cite{Osa21}, and therefore it is
not lattice-induced.
\end{remark}

\begin{corollary}
\label{cor:stealthy-point-process-lattice-induced}
Every ergodic stealthy translation-invariant point process on
$\mathbb R$ with positive intensity and local second moments is
lattice-induced.
\end{corollary}

In particular, the translation action of such a point process is
not weakly mixing, and hence is not mixing.


\subsection{Cross-sections}
\label{subsec:stealthy-cross-section-lattice-induction}

\begin{corollary}
\label{cor:spectral-cross-section-lattice-induction}
Let $(X,\mu,Y)$ be an ergodic $\mathbb R$-cross-section system.
If
\[
    \int_{0<|\xi|<1}
    \frac{1}{\xi^2}\,d\sigma_{Y,\mu}(\xi)<\infty,
\]
then $(X,\mu)$ is lattice-induced.
\end{corollary}

\begin{proof}
The return-time process $\eta_{Y,\mu}$ is ergodic and has positive
intensity. By
Theorem~\ref{thm:spectral-lattice-induction}, its translation
action has a nonzero eigenfunction. Pulling this eigenfunction back
along the equivariant map
\[
    \kappa_Y:X\longrightarrow\mathcal N_s(\mathbb R)
\]
gives a nonzero eigenfunction of the action on $(X,\mu)$.
Lemma~\ref{lem:independent-eigenvalues-lattice-induction}, in
dimension one, now gives the conclusion.
\end{proof}

\begin{corollary}
\label{cor:stealthy-cross-section-lattice-induced}
Every ergodic stealthy $\mathbb R$-cross-section system is
lattice-induced.
\end{corollary}

Neither generation nor the Delone property is required.

\begin{question}[Higher-dimensional converse]
\label{que:higher-dimensional-converse}
Let $d\geq2$, and let $\eta$ be an ergodic stealthy point process
on $\mathbb R^d$. Must its translation action have a nonzero
eigenvalue? In particular, can the translation action of a
stealthy point process be weakly mixing? Under what additional
spectral hypotheses does lattice induction follow?
\end{question}

The one-dimensional proof produces a nonzero eigenvalue from an
interval-count cocycle. In dimension one this already implies
lattice induction, whereas in higher dimensions lattice induction
requires $d$ linearly independent eigenvalues.



\section{The one-dimensional construction}
\label{sec:one-dimensional-construction}

Let $T$ be a Borel automorphism of a standard Borel space $Z$.
We encode the orbit of $z\in Z$ into a band-limited entire
function $H_z$ and perturb a fixed periodic entire function by a
small multiple of $H_z$. The resulting real zero sets form a
uniformly Delone equivariant family. Two frequencies in the
periodic term ensure orbit separation, while a contour calculation
for the logarithmic derivative gives the Fourier identity near the
origin.

The analytic construction was influenced by Benedicks' work on
spectral gaps \cite{Ben84}. Meyer's discussion of mean-periodic
convolution equations and mild measures \cite{Mey21} also
motivated the perturbation of the periodic entire function $S$
while preserving control of its real zero set.

\begin{theorem}[One-dimensional coding theorem]
\label{thm:one-dimensional-coding}
Let $T:Z\to Z$ be a Borel automorphism of a standard Borel space.
There is a Borel family of sets
\[
    z\longmapsto P_z\subset\mathbb R
\]
with the following properties:
\begin{enumerate}
\item
the sets $P_z$ are uniformly Delone;

\item
\[
    P_{Tz}=P_z-1
    \qquad
    (z\in Z);
\]

\item
if $z,w\in Z$ and $u\in\mathbb R$ satisfy
\[
    P_w=P_z-u,
\]
then
\[
    u\in\mathbb Z
    \qquad\text{and}\qquad
    w=T^u z;
\]

\item
for every $z\in Z$,
\[
    \widehat{\delta}_{P_z}
    =
    4\delta_0
    \qquad\text{on }(-1,1)
\]
in the sense of tempered distributions.
\end{enumerate}
\end{theorem}

\subsection{Band-limited orbit coding}
\label{subsec:band-limited-orbit-coding}

We first construct the band-limited entire functions carrying the
orbit information.

\begin{lemma}
\label{lem:band-limited-orbit-coding}
Fix
\[
    \frac12<\sigma<1.
\]
There exist $\tau\in(\frac12,\sigma)$ and a real entire function
$\psi$ such that
\[
    \psi(n)
    =
    \begin{cases}
        1,&n=0,\\
        0,&n\in\mathbb Z\setminus\{0\},
    \end{cases}
\]
and, for every $j,N\geq0$,
\[
    \left|\psi^{(j)}(x+iy)\right|
    \leq
    C_{j,N}
    e^{2\pi\sigma|y|}
    (1+|x|)^{-N}.
\]
Moreover,
\[
    \operatorname{supp}\widehat\psi
    \subset[-\tau,\tau].
\]

Let $b:Z\to[-1,1]$ be a Borel injection and define
\[
    H_z(w)
    :=
    \sum_{n\in\mathbb Z}
    b(T^nz)\psi(w-n).
\]
The series and all its derivatives converge locally uniformly on
$\mathbb C$. The map
\[
    z\longmapsto H_z
\]
is Borel for the topology of locally uniform convergence, and
\[
    H_z(n)=b(T^nz),
    \qquad
    H_{Tz}(w)=H_z(w+1).
\]
For every $j\geq0$,
\[
    \left|H_z^{(j)}(x+iy)\right|
    \leq
    C_j e^{2\pi\sigma|y|}
\]
uniformly in $z\in Z$ and $x,y\in\mathbb R$. Finally,
\[
    \operatorname{supp}\widehat H_z
    \subset[-\tau,\tau]
    \qquad
    (z\in Z).
\]
\end{lemma}

\begin{proof}
Choose an even nonnegative function
\[
    \chi\in C_c^\infty((-\tau,\tau))
\]
which is positive on $[-1/2,1/2]$, and set
\[
    D(\xi)
    :=
    \sum_{k\in\mathbb Z}\chi(\xi+k),
    \qquad
    \theta(\xi)
    :=
    \frac{\chi(\xi)}{D(\xi)}.
\]
The function $D$ is smooth, positive and $1$-periodic. Hence
$\theta\in C_c^\infty((-\tau,\tau))$ and
\[
    \sum_{k\in\mathbb Z}\theta(\xi+k)=1
    \qquad
    (\xi\in\mathbb R).
\]

Define
\[
    \psi(w)
    :=
    \int_{\mathbb R}
    \theta(\xi)e^{2\pi iw\xi}\,d\xi.
\]
The function $\psi$ is entire and real on $\mathbb R$. For
$n\in\mathbb Z$,
\[
\begin{aligned}
    \psi(n)
    &=
    \int_0^1
    e^{2\pi in\xi}
    \sum_{k\in\mathbb Z}\theta(\xi+k)\,d\xi\\
    &=
    \int_0^1e^{2\pi in\xi}\,d\xi,
\end{aligned}
\]
which proves the first claim of the lemma.

Differentiating under the integral and integrating by parts in
$\xi$ gives, for every $j,N\geq0$,
\[
    \left|\psi^{(j)}(x+iy)\right|
    \leq
    C_{j,N}
    e^{2\pi\sigma|y|}
    (1+|x|)^{-N}.
\]
Here the passage from $\tau$ to $\sigma$ absorbs the polynomial
factors in $|y|$ produced by integration by parts. Since
$\widehat\psi=\theta$,
\[
    \operatorname{supp}\widehat\psi
    \subset[-\tau,\tau].
\]

Taking $N>1$ and summing over $n\in\mathbb Z$ shows that the series
defining $H_z$, together with all its derivatives, converges
locally uniformly. Moreover,
\[
\begin{aligned}
    \left|H_z^{(j)}(x+iy)\right|
    &\leq
    C_{j,N}e^{2\pi\sigma|y|}
    \sum_{n\in\mathbb Z}(1+|x-n|)^{-N}\\
    &\leq
    C_j e^{2\pi\sigma|y|}.
\end{aligned}
\]
Its partial sums are Borel maps into the space of entire functions
with the topology of locally uniform convergence, and hence so is
their limit $z\mapsto H_z$.

The interpolation property gives
\[
    H_z(n)=b(T^nz),
\]
while reindexing the series gives
\[
\begin{aligned}
    H_{Tz}(w)
    &=
    \sum_{n\in\mathbb Z}
    b(T^{n+1}z)\psi(w-n)\\
    &=
    \sum_{m\in\mathbb Z}
    b(T^mz)\psi(w+1-m)
    =
    H_z(w+1).
\end{aligned}
\]

For each $z\in Z$, define the signed atomic measure
\[
    \mu_z
    :=
    \sum_{n\in\mathbb Z}b(T^nz)\delta_n.
\]
Since the coefficients $b(T^nz)$ are uniformly bounded, $\mu_z$
is translation bounded and hence defines a tempered distribution.
Moreover, $\psi\in\mathcal S(\mathbb R)$, and its convolution with
$\mu_z$ is given pointwise by
\[
\begin{aligned}
    (\psi*\mu_z)(x)
    &:=
    \int_{\mathbb R}\psi(x-y)\,d\mu_z(y)\\
    &=
    \sum_{n\in\mathbb Z}b(T^nz)\psi(x-n)
    =
    H_z(x).
\end{aligned}
\]
The series is absolutely and locally uniformly convergent by the
decay estimates for $\psi$. Thus, as tempered distributions on
$\mathbb R$,
\[
    H_z=\psi*\mu_z.
\]
Taking Fourier transforms gives
\[
    \widehat H_z
    =
    \widehat\psi\,\widehat\mu_z.
\]
Here $\widehat\mu_z$ is a tempered distribution, and the product
with the smooth compactly supported function $\widehat\psi$ is
well defined. Consequently,
\[
    \operatorname{supp}\widehat H_z
    \subset
    \operatorname{supp}\widehat\psi
    \subset[-\tau,\tau].
\]
\end{proof}

We fix $\sigma$, $\tau$, $\psi$, $b$, and the family
$(H_z)_{z\in Z}$ for the remainder of the section.


\subsection{Uniformly Delone zero sets}
\label{subsec:real-zero-sets}

We perturb the $1$-periodic entire function
\[
    S(w):=\sin(4\pi w)+\sin(2\pi w).
\]
The addition formula gives
\[
    S(w)=2\sin(3\pi w)\cos(\pi w).
\]
The zeros of $\sin(3\pi w)$ are the points of
$\frac13\mathbb Z$, while the zeros of $\cos(\pi w)$ are the
points of $\frac12+\mathbb Z$. These two sets are disjoint, and
therefore the zero set of $S$ in the whole complex plane is
\[
    \Lambda
    :=
    \frac13\mathbb Z
    \,\cup\,
    \left(\frac12+\mathbb Z\right).
\]
Each zero is simple: the vanishing factor has a simple zero
there, while the other factor is nonzero.

For $\varepsilon>0$, define
\[
    F_z(w):=S(w)+\varepsilon H_z(w).
\]

\begin{proposition}
\label{prop:real-zero-family}
For some $\varepsilon>0$, the following assertions hold.

For every $z\in Z$, all zeros of $F_z$ are real and simple, and
its zero set
\[
    P_z:=\{x\in\mathbb R:F_z(x)=0\}
\]
is uniformly Delone. More precisely, there is $\delta>0$ such
that
\[
    P_z=\{p_\lambda(z):\lambda\in\Lambda\},
    \qquad
    |p_\lambda(z)-\lambda|<\delta
\]
for every $\lambda\in\Lambda$ and $z\in Z$.

The map
\[
    z\longmapsto\delta_{P_z}
\]
is Borel, and
\[
    P_{Tz}=P_z-1
    \qquad
    (z\in Z).
\]
\end{proposition}

\begin{remark}
The argument is a concrete sine-type perturbation argument:
$S$ has simple uniformly separated real zeros and dominates the
perturbations $H_z$ in exponential type. See
\cite{Lev96,You01} for the surrounding theory of sine-type
functions and complete interpolating sequences.
\end{remark}

\begin{proof}
Set
\[
    d_\Lambda
    :=
    \inf\{
        |\lambda-\lambda'|:
        \lambda,\lambda'\in\Lambda,\ \lambda\neq\lambda'
    \}>0.
\]
Choose
\[
    0<\delta<\frac14d_\Lambda
\]
so small that each closed disc
\[
    \overline{D(\lambda,\delta)},
    \qquad
    \lambda\in\Lambda,
\]
contains no zero of $S$ other than $\lambda$. These discs are
pairwise disjoint.

There is a constant $c>0$ such that
\[
    |S(x+iy)|
    \geq
    ce^{4\pi|y|}
\]
whenever
\[
    x+iy
    \notin
    \bigcup_{\lambda\in\Lambda}D(\lambda,\delta).
\]
Indeed, for sufficiently large $|y|$, the term
$\sin(4\pi w)$ dominates $\sin(2\pi w)$ uniformly in
$\operatorname{Re}w$. On a bounded horizontal strip, the assertion
follows from periodicity and compactness after removing the discs
around the zeros.

By Lemma~\ref{lem:band-limited-orbit-coding},
\[
    |H_z(x+iy)|
    \leq
    Ce^{2\pi\sigma|y|}
\]
uniformly in $z$, where $\sigma<1$. We may therefore choose
$\varepsilon>0$ so small that
\[
    \varepsilon|H_z(w)|
    \leq
    \frac12|S(w)|
\]
for every $z\in Z$ and every
\[
    w\notin
    \bigcup_{\lambda\in\Lambda}D(\lambda,\delta).
\]
Thus $F_z$ has no zeros outside these discs.

Fix $\lambda\in\Lambda$. The function $S$ has exactly one zero
in $D(\lambda,\delta)$, namely $\lambda$, and this zero has
multiplicity one. On the boundary of the disc,
\[
    |\varepsilon H_z(w)|<|S(w)|.
\]
Rouché's theorem therefore implies that $F_z$ and $S$ have the
same number of zeros in $D(\lambda,\delta)$, counted with
multiplicity. Hence $F_z$ has precisely one zero in the disc,
and its total multiplicity is one. In particular, this zero is
simple.

Moreover,
\[
    F_z(\overline w)=\overline{F_z(w)},
\]
and the disc $D(\lambda,\delta)$ is invariant under complex
conjugation. If its unique zero were nonreal, its conjugate would
be a distinct second zero in the same disc. Thus the zero is real.
We denote it by $p_\lambda(z)$.

The estimate
\[
    |p_\lambda(z)-\lambda|<\delta
\]
implies uniform separation, since
\[
    |p_\lambda(z)-p_{\lambda'}(z)|
    \geq
    d_\Lambda-2\delta
    \qquad
    (\lambda\neq\lambda').
\]
Uniform relative density follows from the bounded gaps of
$\Lambda$. Hence the sets $P_z$ are uniformly Delone.

The map $z\mapsto F_z$ is Borel for the topology of locally
uniform convergence. Moreover,
\[
    p_\lambda(z)
    =
    \frac{1}{2\pi i}
    \int_{\partial D(\lambda,\delta)}
    w\frac{F_z'(w)}{F_z(w)}\,dw,
\]
so each map $z\mapsto p_\lambda(z)$ is Borel. For every $f\in C_c(\mathbb R)$, only 
finitely many $\lambda\in\Lambda$ satisfy
\[
    D(\lambda,\delta)\cap\operatorname{supp}f\neq\varnothing.
\]
Consequently,
\[
    \int_{\mathbb R}f\,d\delta_{P_z}
    =
    \sum_{\lambda\in\Lambda}f(p_\lambda(z))
\]
is a finite sum of Borel functions. Since the vague Borel
structure is generated by these evaluation maps,
\[
    z\longmapsto\delta_{P_z}
\]
is Borel.

Finally, the periodicity of $S$ and the identity
\[
    H_{Tz}(w)=H_z(w+1)
\]
give
\[
    F_{Tz}(w)=F_z(w+1),
\]
and therefore
\[
    P_{Tz}=P_z-1.
\]
\end{proof}

We fix such an $\varepsilon$ for the remainder of the section.


\subsection{The Fourier identity}
\label{subsec:Fourier-identity}

We prove that the counting measures of the zero sets satisfy the
same Fourier identity near the origin. By
Lemma~\ref{lem:band-limited-orbit-coding},
\[
    \operatorname{supp}\widehat H_z
    \subset[-\tau,\tau],
    \qquad
    \tau<1.
\]

For $\operatorname{Im}w>0$, write
\[
    F_z(w)
    =
    -\frac{e^{-4\pi i w}}{2i}
    \bigl(1+G_z(w)\bigr),
\]
where
\[
    G_z(w)
    :=
    e^{2\pi i w}-e^{6\pi i w}-e^{8\pi i w}
    -2i\varepsilon e^{4\pi i w}H_z(w).
\]
The bounds from
Lemma~\ref{lem:band-limited-orbit-coding} give
\[
    \sup_{z\in Z}\sup_{x\in\mathbb R}
    |G_z(x+iY)|
    \longrightarrow0
    \qquad
    (Y\longrightarrow+\infty).
\]
Indeed, the exponential terms tend uniformly to zero, while
\[
    \left|
        e^{4\pi i(x+iY)}H_z(x+iY)
    \right|
    \leq
    C e^{-2\pi(2-\sigma)Y}.
\]
Fix $Y>0$ so large that
\[
    \sup_{z\in Z}\sup_{x\in\mathbb R}
    |G_z(x+iY)|<\frac12.
\]

\begin{lemma}
\label{lem:logarithmic-derivative-support}
For every $z\in Z$, define
\[
    A_z^+(x)
    :=
    \frac{F_z'}{F_z}(x+iY)+4\pi i,
    \qquad
    A_z^-(x)
    :=
    \frac{F_z'}{F_z}(x-iY)-4\pi i.
\]
Then
\[
    \operatorname{supp}\widehat A_z^+
    \subset[1,\infty),
    \qquad
    \operatorname{supp}\widehat A_z^-
    \subset(-\infty,-1].
\]
\end{lemma}

\begin{proof}
On the upper horizontal line, the factorization
\[
    F_z(w)
    =
    -\frac{e^{-4\pi iw}}{2i}\bigl(1+G_z(w)\bigr)
\]
gives
\[
    \frac{F_z'}{F_z}(w)
    =
    -4\pi i+\frac{G_z'(w)}{1+G_z(w)}.
\]
Consequently,
\[
    A_z^+(x)
    =
    \frac{G_z'(x+iY)}{1+G_z(x+iY)}.
\]

Since
\[
    \sup_{x\in\mathbb R}|G_z(x+iY)|<\frac12,
\]
the values of $1+G_z(x+iY)$ lie in the disc
$D(1,1/2)$, which does not meet the origin. We may therefore use
the branch of the logarithm defined there by the absolutely
convergent power series
\[
    \log(1+\zeta)
    =
    \sum_{n=1}^{\infty}
    \frac{(-1)^{n+1}}{n}\zeta^n,
    \qquad
    |\zeta|<1.
\]
Set
\[
    L_z(x)
    :=
    \log\bigl(1+G_z(x+iY)\bigr).
\]
Differentiating with respect to $x$ gives
\[
    L_z'(x)
    =
    \frac{G_z'(x+iY)}{1+G_z(x+iY)}
    =
    A_z^+(x).
\]
Thus \(A_z^+\) is the derivative of the logarithmic correction to
the dominant exponential factor of \(F_z\).

We now examine the Fourier support of \(L_z\). From the explicit
formula for \(G_z\),
\[
    \operatorname{supp}
    \widehat{G_z(\,\cdot+iY)}
    \subset
    \{1,3,4\}\cup[2-\tau,2+\tau]
    \subset[1,\infty),
\]
because \(\tau<1\). Products correspond to convolution on the
Fourier side, and hence
\[
    \operatorname{supp}
    \widehat{G_z(\,\cdot+iY)^n}
    \subset[n,\infty)
    \qquad
    (n\geq1).
\]
The logarithmic series therefore contains only positive
frequencies. More precisely, its partial sums
\[
    L_{z,N}
    :=
    \sum_{n=1}^{N}
    \frac{(-1)^{n+1}}{n}
    G_z(\,\cdot+iY)^n
\]
satisfy
\[
    \operatorname{supp}\widehat L_{z,N}
    \subset[1,\infty).
\]
The series defining \(L_z\), together with its derivatives,
converges uniformly on the line \(\mathbb R+iY\). In particular,
\(L_{z,N}\to L_z\) in \(\mathcal S'(\mathbb R)\), and therefore
\[
    \operatorname{supp}\widehat L_z
    \subset[1,\infty).
\]

Finally,
\[
    \widehat{L_z'}(\xi)
    =
    2\pi i\xi\,\widehat L_z(\xi).
\]
Differentiation does not enlarge Fourier support. Since
\[
    \operatorname{supp}\widehat L_z\subset[1,\infty)
\]
and $A_z^+=L_z'$, it follows that
\[
    \operatorname{supp}\widehat A_z^+
    \subset[1,\infty).
\]

Because \(F_z\) is real on the real axis,
\[
    F_z(\overline w)=\overline{F_z(w)},
    \qquad
    F_z'(\overline w)=\overline{F_z'(w)}.
\]
It follows that
\[
    A_z^-(x)=\overline{A_z^+(x)}.
\]
For any tempered distribution \(A\),
\[
    \widehat{\overline A}(\xi)
    =
    \overline{\widehat A(-\xi)},
\]
so complex conjugation reflects Fourier support through the
origin. Hence
\[
    \operatorname{supp}\widehat A_z^-
    \subset(-\infty,-1].
\]
\end{proof}

\begin{proposition}
\label{prop:Fourier-identity}
For every $z\in Z$,
\[
    \widehat{\delta}_{P_z}
    =
    4\delta_0
    \qquad\text{on }(-1,1)
\]
in the sense of tempered distributions.
\end{proposition}

\begin{proof}
It suffices to prove that
\[
    \sum_{p\in P_z}f(p)
    =
    4\int_{\mathbb R}f(x)\,dx
\]
for every $f\in\mathcal S(\mathbb R)$ satisfying
\[
    \operatorname{supp}\widehat f\subset(-1,1).
\]
Such a function extends to an entire function and is rapidly
decreasing in every horizontal strip.

Choose $\beta\in[0,1)$ such that
\[
    \operatorname{dist}(\beta+\mathbb Z,\Lambda)>\delta.
\]
For $N\geq1$, let $R_N$ be the rectangle with vertical sides
\[
    \operatorname{Re}w=\beta-N,
    \qquad
    \operatorname{Re}w=\beta+N
\]
and horizontal sides contained in $\mathbb R\pm iY$. Apply the
residue theorem to
\[
    w\longmapsto
    f(w)\frac{F_z'(w)}{F_z(w)}
\]
on $R_N$.

The vertical sides lie outside all the discs
$D(\lambda,\delta)$. By the choice of $\varepsilon$ in
Proposition~\ref{prop:real-zero-family},
\[
    |F_z(w)|
    \geq
    |S(w)|-\varepsilon|H_z(w)|
    \geq
    \frac12|S(w)|
\]
there. Since $S$ is $1$-periodic and has no zeros on these
vertical segments, compactness gives
\[
    |F_z(w)|\geq c_Y>0
\]
uniformly in $z$ and $N$. The derivatives $F_z'$ are also
uniformly bounded on the strip $|\operatorname{Im}w|\leq Y$.
Hence
\[
    \left|\frac{F_z'(w)}{F_z(w)}\right|\leq C_Y
\]
on the vertical sides. The rapid decay of $f$ in this strip
therefore implies that the two vertical integrals tend to zero as
$N\to\infty$.

All zeros of $F_z$ are real and simple. At a zero $p\in P_z$,
\[
    \operatorname*{Res}_{w=p}
    \left(
        f(w)\frac{F_z'(w)}{F_z(w)}
    \right)
    =
    f(p).
\]
Moreover, the uniform separation of $P_z$ and the rapid decay of
$f$ imply that
\[
    \sum_{p\in P_z}|f(p)|<\infty.
\]
Letting $N\to\infty$ in the residue formula gives
\[
\begin{aligned}
    \sum_{p\in P_z}f(p)
    =
    \frac{1}{2\pi i}
    \bigg(
        &\int_{\mathbb R}
        f(x-iY)\frac{F_z'}{F_z}(x-iY)\,dx\\
        &-
        \int_{\mathbb R}
        f(x+iY)\frac{F_z'}{F_z}(x+iY)\,dx
    \bigg).
\end{aligned}
\]

Put
\[
    f_Y^\pm(x):=f(x\pm iY).
\]
Then
\[
    \widehat{f_Y^\pm}(\xi)
    =
    e^{\mp2\pi Y\xi}\widehat f(\xi),
\]
and hence
\[
    \operatorname{supp}\widehat{f_Y^\pm}
    \subset(-1,1).
\]
By Lemma~\ref{lem:logarithmic-derivative-support},
\[
    \operatorname{supp}\widehat A_z^+
    \subset[1,\infty),
    \qquad
    \operatorname{supp}\widehat A_z^-
    \subset(-\infty,-1].
\]
Since
\[
    \int_{\mathbb R}g(x)A(x)\,dx
    =
    \left\langle
        \widehat A,\widehat g(-\,\cdot)
    \right\rangle,
\]
the disjointness of the Fourier supports gives
\[
    \int_{\mathbb R}f(x+iY)A_z^+(x)\,dx=0,
    \qquad
    \int_{\mathbb R}f(x-iY)A_z^-(x)\,dx=0.
\]

Also,
\[
    \int_{\mathbb R}f(x\pm iY)\,dx
    =
    \widehat{f_Y^\pm}(0)
    =
    \widehat f(0)
    =
    \int_{\mathbb R}f(x)\,dx.
\]
Substituting
\[
    \frac{F_z'}{F_z}(x+iY)
    =
    -4\pi i+A_z^+(x),
    \qquad
    \frac{F_z'}{F_z}(x-iY)
    =
    4\pi i+A_z^-(x)
\]
into the residue formula yields
\[
    \sum_{p\in P_z}f(p)
    =
    4\int_{\mathbb R}f(x)\,dx.
\]
\end{proof}


\subsection{Orbit recovery}
\label{subsec:orbit-recovery}

We now show that a translate of one zero set can coincide with
another only when the translation comes from the underlying
$\mathbb Z$-action.

\begin{proposition}
\label{prop:recovery-coding}
Let $z,w\in Z$ and $u\in\mathbb R$. If
\[
    P_w=P_z-u,
\]
then
\[
    u\in\mathbb Z
    \qquad\text{and}\qquad
    w=T^u z.
\]
\end{proposition}

\begin{proof}
The entire functions $F_w$ and $F_z(\,\cdot+u)$ have the same
zeros. By Proposition~\ref{prop:real-zero-family}, all these zeros
are simple, so the two functions have the same zero divisor.
Moreover, both are entire functions of order at most one.
Hadamard factorization \cite[Chapter~5, Theorem~5.1]{SS03} therefore gives constants
$a,b\in\mathbb C$ such that
\[
    F_z(v+u)=e^{av+b}F_w(v)
    \qquad
    (v\in\mathbb C).
\]

We first show that the exponential factor is constant. Recall that
\[
    \operatorname{supp}\widehat H_z\subset[-\tau,\tau],
    \qquad
    \tau<1.
\]
Thus
\[
    H_z(x+iy)=O(e^{2\pi\tau y})
    \qquad
    (y\longrightarrow+\infty),
\]
uniformly in $x$ and $z$. Since
\[
    F_z(v)=\sin(4\pi v)+\sin(2\pi v)+\varepsilon H_z(v),
\]
the term $\sin(4\pi v)$ dominates on the upper imaginary axis.
More precisely,
\[
    e^{-4\pi y}F_z(u+iy)
    \longrightarrow
    -\frac{e^{-4\pi iu}}{2i},
\]
whereas
\[
    e^{-4\pi y}F_w(iy)
    \longrightarrow
    -\frac1{2i}.
\]
Consequently,
\[
    \frac{F_z(u+iy)}{F_w(iy)}
    \longrightarrow
    e^{-4\pi iu}
    \qquad
    (y\longrightarrow+\infty).
\]

On the other hand, the factorization identity gives
\[
    \frac{F_z(u+iy)}{F_w(iy)}
    =
    e^{aiy+b}.
\]
Since this expression has a finite nonzero limit as
$y\to+\infty$, we must have $a=0$. Hence
\[
    F_z(v+u)=cF_w(v)
    \qquad
    (v\in\mathbb C)
\]
for some $c\in\mathbb C\setminus\{0\}$.

We now compare the isolated frequencies contributed by the
periodic term. On the real line,
\[
    \widehat S
    =
    \frac{1}{2i}
    \bigl(
        \delta_2-\delta_{-2}
        +\delta_1-\delta_{-1}
    \bigr),
\]
whereas $\widehat H_z$ and $\widehat H_w$ are supported in
$[-\tau,\tau]$. Since $\tau<1$, the perturbation terms contribute
nothing at the frequencies $1$ and $2$. Taking Fourier transforms
of
\[
    F_z(\,\cdot+u)=cF_w
\]
and comparing the coefficients at these two frequencies gives
\[
    e^{2\pi iu}=c,
    \qquad
    e^{4\pi iu}=c.
\]
Therefore
\[
    e^{2\pi iu}=e^{4\pi iu},
\]
and hence $ e^{2\pi iu}=1$. Thus $u\in\mathbb Z$ and $c=1$.

Since $S$ is $1$-periodic, the identity
\[
    F_z(v+u)=F_w(v)
\]
now reduces to
\[
    H_z(v+u)=H_w(v)
    \qquad
    (v\in\mathbb C).
\]
Evaluating at $v=0$ and using
Lemma~\ref{lem:band-limited-orbit-coding}, we obtain
\[
    b(T^u z)
    =
    H_z(u)
    =
    H_w(0)
    =
    b(w).
\]
The injectivity of $b$ gives
\[
    w=T^u z.
\]
\end{proof}

The family $(P_z)_{z\in Z}$ therefore satisfies the
orbit-separation condition of
Proposition~\ref{prop:equivariant-families}. Together with
Propositions~\ref{prop:real-zero-family} and
\ref{prop:Fourier-identity}, this proves
Theorem~\ref{thm:one-dimensional-coding}.


\subsection{Passage to induced spaces}
\label{subsec:one-dimensional-induced-space}

Let $a>0$, put $\Gamma:=a\mathbb Z$, and let $(Z,\nu)$ be a
probability-preserving Borel $\Gamma$-space. Apply
Theorem~\ref{thm:one-dimensional-coding} to the Borel
automorphism
\[
    Tz:=a.z,
\]
and define
\[
    Q_z:=aP_z.
\]

\begin{corollary}
\label{cor:one-dimensional-induced-space}
The family $(Q_z)_{z\in Z}$ is Borel, $\Gamma$-equivariant and
uniformly Delone. It satisfies
\[
    \widehat{\delta}_{Q_z}
    =
    \frac{4}{a}\delta_0
    \qquad\text{on }
    \left(-\frac1a,\frac1a\right),
\]
and
\[
    Q_w=Q_z-u
    \quad\Longrightarrow\quad
    u\in\Gamma
    \ \text{ and }\
    w=u.z.
\]
Consequently, $(\operatorname{Ind}_{\Gamma}^{\mathbb R}(Z),\mu_{\Gamma,Z})$
admits a generating stealthy Delone cross-section.
\end{corollary}

\begin{proof}
For $\gamma=ka\in\Gamma$, the equivariance relation
$P_{Tz}=P_z-1$ gives
\[
    Q_{\gamma.z}
    =
    aP_{T^kz}
    =
    a(P_z-k)
    =
    Q_z-\gamma.
\]
Borelness and uniform Delone bounds are preserved by the dilation
$x\mapsto ax$.

If $Q_w=Q_z-u$, then
\[
    P_w=P_z-\frac{u}{a}.
\]
Proposition~\ref{prop:recovery-coding} gives
\[
    \frac{u}{a}\in\mathbb Z
    \qquad\text{and}\qquad
    w=T^{u/a}z=u.z.
\]

To obtain the Fourier identity, let
$f\in\mathcal S(\mathbb R)$ satisfy
\[
    \operatorname{supp}\widehat f
    \subset
    \left(-\frac1a,\frac1a\right),
\]
and put $g(x):=f(ax)$. Then
\[
    \widehat g(\xi)
    =
    \frac1a\widehat f\left(\frac{\xi}{a}\right),
\]
so
\[
    \operatorname{supp}\widehat g\subset(-1,1).
\]
Proposition~\ref{prop:Fourier-identity} therefore gives
\[
\begin{aligned}
    \sum_{q\in Q_z}f(q)
    &=
    \sum_{p\in P_z}f(ap)
    =
    \sum_{p\in P_z}g(p)
    =
    4\int_{\mathbb R}g(x)\,dx
    =
    \frac4a\int_{\mathbb R}f(x)\,dx.
\end{aligned}
\]
The final assertion follows from
Propositions~\ref{prop:equivariant-families} and
\ref{prop:stealthiness-criterion}.
\end{proof}

\begin{corollary}[One-dimensional characterization]
\label{cor:one-dimensional-characterization}
Let $(X,\mu)$ be an ergodic probability-preserving Borel
$\mathbb R$-space. Then the following are equivalent:
\begin{enumerate}
\item
$(X,\mu)$ is lattice-induced;

\item
$(X,\mu)$ admits a generating stealthy Delone cross-section.
\end{enumerate}
\end{corollary}

\begin{proof}
The forward implication follows from
Corollary~\ref{cor:one-dimensional-induced-space}, and the reverse
implication from
Corollary~\ref{cor:stealthy-cross-section-lattice-induced}.
\end{proof}


\section{The layered construction}
\label{sec:layered-construction}

We extend the one-dimensional construction by placing a
one-dimensional coding set on each line
\[
    \mathbb R\times\{m\},
    \qquad
    m\in\mathbb Z^{d-1}.
\]
The code on the layer indexed by $m$ is evaluated at
$(0,m).z$. With this choice, translation by $\mathbb Z^d$
simultaneously translates the one-dimensional sets and permutes
the layers. The Fourier identity in the first coordinate is then
combined with Poisson summation in the remaining coordinates.

We first work with the lattice $\mathbb Z^d$. An arbitrary lattice
will be treated at the end of the section by a linear change of
coordinates.

\begin{theorem}
\label{thm:higher-dimensional-coding}
Let $d\geq2$, and let $Z$ be a standard Borel space equipped with
a Borel action of $\mathbb Z^d$. There is a Borel family
$z\mapsto P_z\subset\mathbb R^d$ such that:
\begin{enumerate}
\item
the sets $P_z$ are uniformly Delone;

\item
\[
    P_{n.z}=P_z-n
    \qquad
    (n\in\mathbb Z^d,\ z\in Z);
\]

\item
if
\[
    P_w=P_z-u,
    \qquad
    z,w\in Z,\quad u\in\mathbb R^d,
\]
then
\[
    u\in\mathbb Z^d
    \qquad\text{and}\qquad
    w=u.z;
\]

\item
for every $z\in Z$,
\[
    \widehat{\delta}_{P_z}
    =
    4\delta_0
    \qquad\text{on }(-1,1)^d
\]
in the sense of tempered distributions.
\end{enumerate}
\end{theorem}

\subsection{The layered family}
\label{subsec:layered-family}

Write
\[
    \mathbb R^d
    =
    \mathbb R\times\mathbb R^{d-1},
    \qquad
    \mathbb Z^d
    =
    \mathbb Z\times\mathbb Z^{d-1},
\]
and let
\[
    Tz:=e_1.z.
\]
Apply Theorem~\ref{thm:one-dimensional-coding} to the Borel
automorphism $T$, and denote the resulting family by
$(L_z)_{z\in Z}$. Thus the sets $L_z\subset\mathbb R$ are
uniformly Delone and
\[
    L_{T^kz}=L_z-k
    \qquad
    (k\in\mathbb Z,\ z\in Z).
\]

For $m\in\mathbb Z^{d-1}$, put
\[
    \widetilde m:=(0,m)\in\mathbb Z^d
\]
and define
\[
    P_z
    :=
    \bigcup_{m\in\mathbb Z^{d-1}}
    \left(
        L_{\widetilde m.z}\times\{m\}
    \right).
\]
Thus the layer of $P_z$ over $m$ is the one-dimensional code
associated with $\widetilde m.z$.

\begin{proposition}
\label{prop:layered-family}
The map
\[
    z\longmapsto\delta_{P_z}
\]
is Borel, the family $(P_z)_{z\in Z}$ is
$\mathbb Z^d$-equivariant, and the sets $P_z$ are uniformly
Delone.
\end{proposition}

\begin{proof}
Let $f\in C_c(\mathbb R^d)$. Only finitely many
$m\in\mathbb Z^{d-1}$ meet the projection of
$\operatorname{supp}f$ onto the last $d-1$ coordinates. Hence
\[
    \int_{\mathbb R^d}f\,d\delta_{P_z}
    =
    \sum_{m\in\mathbb Z^{d-1}}
    \int_{\mathbb R}
    f(t,m)\,d\delta_{L_{\widetilde m.z}}(t),
\]
where the sum is finite. Each summand is Borel in $z$, and
therefore $z\mapsto\delta_{P_z}$ is Borel.

Let
\[
    n=(k,\ell)
    \in
    \mathbb Z\times\mathbb Z^{d-1}.
\]
Since the action of $\mathbb Z^d$ is abelian,
\[
    \widetilde m.(n.z)
    =
    T^k\bigl(\widetilde{m+\ell}.z\bigr).
\]
The equivariance of the one-dimensional family gives
\[
    L_{\widetilde m.(n.z)}
    =
    L_{\widetilde{m+\ell}.z}-k.
\]
Consequently,
\[
\begin{aligned}
    P_{n.z}
    &=
    \bigcup_{m\in\mathbb Z^{d-1}}
    \left(
        (L_{\widetilde{m+\ell}.z}-k)
        \times\{m\}
    \right)\\
    &=
    P_z-(k,\ell)
    =
    P_z-n,
\end{aligned}
\]
where the second equality follows by reindexing the layers.

Choose $r,R>0$ such that every $L_z$ is $r$-separated and
$R$-relatively dense. Two points of $P_z$ on the same layer are
separated by at least $r$, while points on distinct layers are
separated by at least $1$. Hence every $P_z$ is
$\min\{r,1\}$-separated.

Finally, let
\[
    (x,y)\in\mathbb R\times\mathbb R^{d-1}.
\]
Choose $m\in\mathbb Z^{d-1}$ such that
\[
    \|y-m\|
    \leq
    \frac{\sqrt{d-1}}{2}.
\]
There is a point $p\in L_{\widetilde m.z}$ with
\[
    |x-p|\leq R.
\]
Therefore
\[
    \|(x,y)-(p,m)\|
    \leq
    \left(
        R^2+\frac{d-1}{4}
    \right)^{1/2}.
\]
This bound is independent of $z$, so the family is uniformly
relatively dense.
\end{proof}


\subsection{The Fourier identity}
\label{subsec:layered-Fourier-identity}

The one-dimensional Fourier identity on each layer combines with
Poisson summation in the transverse variables.

\begin{proposition}
\label{prop:layered-Fourier-identity}
For every $z\in Z$,
\[
    \widehat{\delta}_{P_z}
    =
    4\delta_0
    \qquad\text{on }(-1,1)^d
\]
in the sense of tempered distributions.
\end{proposition}

\begin{proof}
Let $f\in\mathcal S(\mathbb R^d)$ satisfy
\[
    \operatorname{supp}\widehat f\subset(-1,1)^d.
\]
Writing points of $\mathbb R^d$ as
$(x,y)\in\mathbb R\times\mathbb R^{d-1}$, we have
\[
    \sum_{p\in P_z}f(p)
    =
    \sum_{m\in\mathbb Z^{d-1}}
    \sum_{t\in L_{\widetilde m.z}}f(t,m).
\]
This sum is absolutely convergent because $P_z$ is uniformly
separated and $f$ is rapidly decreasing.

Fix $m\in\mathbb Z^{d-1}$ and set
\[
    f_m(x):=f(x,m).
\]
Its Fourier transform in the first variable is
\[
    \widehat{f_m}(\xi)
    =
    \int_{\mathbb R^{d-1}}
    \widehat f(\xi,\eta)e^{2\pi i m\cdot\eta}\,d\eta.
\]
It follows that
\[
    \operatorname{supp}\widehat{f_m}\subset(-1,1).
\]
The one-dimensional Fourier identity,
Proposition~\ref{prop:Fourier-identity}, therefore gives
\[
    \sum_{t\in L_{\widetilde m.z}}f(t,m)
    =
    4\int_{\mathbb R}f(x,m)\,dx.
\]
Hence
\[
    \sum_{p\in P_z}f(p)
    =
    4\sum_{m\in\mathbb Z^{d-1}}g(m),
    \qquad
    g(y):=\int_{\mathbb R}f(x,y)\,dx.
\]

The function $g$ belongs to $\mathcal S(\mathbb R^{d-1})$, and
integrating in the first variable amounts to restricting the
Fourier transform to the hyperplane $\xi=0$:
\[
    \widehat g(\eta)=\widehat f(0,\eta).
\]
Consequently,
\[
    \operatorname{supp}\widehat g
    \subset(-1,1)^{d-1}.
\]
Poisson summation now gives
\[
    \sum_{m\in\mathbb Z^{d-1}}g(m)
    =
    \sum_{k\in\mathbb Z^{d-1}}\widehat g(k).
\]
Every nonzero point of $\mathbb Z^{d-1}$ lies outside
$(-1,1)^{d-1}$, so only the term $k=0$ remains:
\[
    \sum_{m\in\mathbb Z^{d-1}}g(m)
    =
    \widehat g(0)
    =
    \int_{\mathbb R^{d-1}}g(y)\,dy.
\]
It follows that
\[
    \sum_{p\in P_z}f(p)
    =
    4\int_{\mathbb R^d}f(x,y)\,dx\,dy.
\]
This is precisely
\[
    \widehat{\delta}_{P_z}
    =
    4\delta_0
    \qquad\text{on }(-1,1)^d.
\]
\end{proof}

\subsection{Orbit recovery and arbitrary lattices}
\label{subsec:layered-generation}

We first verify the orbit-separation condition for the layered
family.

\begin{proposition}
\label{prop:layered-generation}
Let $z,w\in Z$ and $u\in\mathbb R^d$. If
\[
    P_w=P_z-u,
\]
then
\[
    u\in\mathbb Z^d
    \qquad\text{and}\qquad
    w=u.z.
\]
\end{proposition}

\begin{proof}
Write
\[
    u=(s,v)\in\mathbb R\times\mathbb R^{d-1},
\]
and let
\[
    \pi:\mathbb R^d\longrightarrow\mathbb R^{d-1}
\]
be projection onto the last $d-1$ coordinates. Every layer of
$P_z$ is nonempty, so
\[
    \pi(P_z)=\mathbb Z^{d-1}
    \qquad
    (z\in Z).
\]
Applying $\pi$ to the identity $P_w=P_z-u$ gives
\[
    \mathbb Z^{d-1}
    =
    \mathbb Z^{d-1}-v.
\]
Hence $v\in\mathbb Z^{d-1}$.

The layer of $P_w$ over $0$ is $L_w$. The layer of $P_z-u$ over
$0$ comes from the layer of $P_z$ over $v$, and is therefore
\[
    L_{\widetilde v.z}-s.
\]
Thus
\[
    L_w=L_{\widetilde v.z}-s.
\]
Applying Proposition~\ref{prop:recovery-coding} to the
automorphism $Tz=e_1.z$, we obtain
\[
    s\in\mathbb Z
    \qquad\text{and}\qquad
    w=T^s(\widetilde v.z).
\]
Since the $\mathbb Z^d$-action is abelian,
\[
    T^s(\widetilde v.z)=(s,v).z=u.z.
\]
Thus $u\in\mathbb Z^d$ and $w=u.z$.
\end{proof}

Propositions~\ref{prop:layered-family},
\ref{prop:layered-Fourier-identity}, and
\ref{prop:layered-generation} prove
Theorem~\ref{thm:higher-dimensional-coding}.

We now transfer the construction from $\mathbb Z^d$ to an
arbitrary lattice. Let
\[
    \Gamma<\mathbb R^d
\]
be a full-rank lattice. Choose an ordered basis
$\gamma_1,\ldots,\gamma_d$ of $\Gamma$, and let
\[
    A:\mathbb R^d\longrightarrow\mathbb R^d
\]
be the linear isomorphism determined by
\[
    Ae_j=\gamma_j,
    \qquad
    1\leq j\leq d.
\]
Then
\[
    \Gamma=A\mathbb Z^d,
    \qquad
    |\det A|=\operatorname{covol}(\Gamma).
\]

\begin{theorem}
\label{thm:lattice-coding}
Let $Z$ be a standard Borel space equipped with a Borel action of
$\Gamma$. There is a Borel $\Gamma$-equivariant uniformly Delone
family $(P_z)_{z\in Z}$ such that
\[
    P_w=P_z-u
    \quad\Longrightarrow\quad
    u\in\Gamma
    \ \text{ and }\
    w=u.z.
\]
Moreover, if
\[
    \Omega_A
    :=
    \left\{
        \xi\in\mathbb R^d:
        A^{\mathsf T}\xi\in(-1,1)^d
    \right\},
\]
then
\[
    \widehat{\delta}_{P_z}
    =
    \frac{4}{\operatorname{covol}(\Gamma)}\delta_0
    \qquad\text{on }\Omega_A
\]
for every $z\in Z$.
\end{theorem}

\begin{proof}
For $d=1$, this is
Corollary~\ref{cor:one-dimensional-induced-space}. Suppose that
$d\geq2$.

Turn $Z$ into a $\mathbb Z^d$-space by defining
\[
    n.z:=(An).z,
    \qquad
    n\in\mathbb Z^d.
\]
Let $(Q_z)_{z\in Z}$ be the family supplied by
Theorem~\ref{thm:higher-dimensional-coding}, and set
\[
    P_z:=AQ_z.
\]

If $\gamma=An\in\Gamma$, then
\[
    P_{\gamma.z}
    =
    AQ_{n.z}
    =
    A(Q_z-n)
    =
    P_z-\gamma.
\]
Thus the family is $\Gamma$-equivariant. Since $A$ is an
invertible linear map, it also preserves uniform discreteness and
relative density, up to changing the constants. Hence the sets
$P_z$ are uniformly Delone.

If
\[
    P_w=P_z-u,
\]
then
\[
    Q_w=Q_z-A^{-1}u.
\]
Theorem~\ref{thm:higher-dimensional-coding} gives
\[
    A^{-1}u\in\mathbb Z^d
    \qquad\text{and}\qquad
    w=(A^{-1}u).z.
\]
It follows that
\[
    u\in A\mathbb Z^d=\Gamma
    \qquad\text{and}\qquad
    w=u.z.
\]

It remains to transform the Fourier identity. Let
$f\in\mathcal S(\mathbb R^d)$ satisfy
\[
    \operatorname{supp}\widehat f\subset\Omega_A,
\]
and define
\[
    g(x):=f(Ax).
\]
Then
\[
    \widehat g(\xi)
    =
    |\det A|^{-1}
    \widehat f(A^{-\mathsf T}\xi).
\]
The definition of $\Omega_A$ therefore implies
\[
    \operatorname{supp}\widehat g\subset(-1,1)^d.
\]
By Theorem~\ref{thm:higher-dimensional-coding},
\[
\begin{aligned}
    \sum_{p\in P_z}f(p)
    &=
    \sum_{q\in Q_z}f(Aq)
    =
    \sum_{q\in Q_z}g(q)
    =
    4\int_{\mathbb R^d}g(x)\,dx
    =
    \frac{4}{|\det A|}
    \int_{\mathbb R^d}f(x)\,dx.
\end{aligned}
\]
Since $|\det A|=\operatorname{covol}(\Gamma)$, this is equivalent
to
\[
    \widehat{\delta}_{P_z}
    =
    \frac{4}{\operatorname{covol}(\Gamma)}\delta_0
    \qquad\text{on }\Omega_A.
\]
\end{proof}

\begin{theorem}[Stealthy realization]
\label{thm:lattice-induced-realization}
Every lattice-induced probability-preserving Borel
$\mathbb R^d$-space admits a generating stealthy Delone
cross-section.
\end{theorem}

\begin{proof}
Represent the given space as
$\smash{\operatorname{Ind}_{\Gamma}^{\mathbb R^d}(Z)}$
for a probability-preserving Borel action of a full-rank lattice
$\Gamma$ on $(Z,\nu)$. Let $(P_z)_{z\in Z}$ be the family supplied
by Theorem~\ref{thm:lattice-coding}, and let $Y_P$ be the
associated cross-section.

The family is uniformly Delone and satisfies the orbit-separation
condition, so
Proposition~\ref{prop:equivariant-families} shows that $Y_P$ is
Delone and generating. Its Fourier identity holds on the open
neighborhood $\Omega_A$ of the origin, so
Proposition~\ref{prop:stealthiness-criterion} shows that $Y_P$ is
stealthy.
\end{proof}


\section{Continuous components of the Bartlett spectrum}
\label{sec:continuous-spectral-components}

A generating cross-section need not have Bartlett spectrum
equivalent to the maximal spectral type of the ambient action.
Indeed, generation concerns the full return-time map, whereas the
Bartlett spectrum only records the closed subspace generated by
centered linear statistics. We show, however, that for the
constructions above, a nonzero part of the spectral type coming
from a weakly mixing base can already occur in the Bartlett
spectrum.

\subsection{The induced spectral decomposition}
\label{subsec:induced-spectral-decomposition}

Let $\Gamma\leq\mathbb R^d$ be a lattice. We identify
$\widehat\Gamma$ with $\mathbb R^d/\Gamma^*$ through the
restriction map
\[
    \pi_\Gamma:\mathbb R^d\longrightarrow\widehat\Gamma,
    \qquad
    \pi_\Gamma(\xi)(\gamma)
    =
    e^{2\pi i\langle\gamma,\xi\rangle}.
\]

Choose a basis
\[
    \gamma_1^*,\ldots,\gamma_d^*
\]
of $\Gamma^*$ and let
\[
    Q
    :=
    \left\{
        \sum_{j=1}^d t_j\gamma_j^*:
        0\leq t_j<1
    \right\}.
\]
Every coset in $\mathbb R^d/\Gamma^*$ has a unique representative
in $Q$. The restriction
\[
    \pi_\Gamma|_Q:Q\longrightarrow\widehat\Gamma
\]
is therefore a Borel bijection. Its inverse is Borel by the
Lusin--Souslin theorem
\cite[Theorem~15.1]{Kec95}. We denote this inverse by 
$\smash{\mathfrak s_\Gamma:\widehat\Gamma\longrightarrow Q}$.

Let $\mathfrak m$ be a finite measure on $\widehat\Gamma$. Its
$\Gamma^*$-periodic lift is the locally finite measure
$\widetilde{\mathfrak m}$ on $\mathbb R^d$ defined by
\[
    \int_{\mathbb R^d}f(\xi)\,
    d\widetilde{\mathfrak m}(\xi)
    =
    \int_{\widehat\Gamma}
    \sum_{\gamma^*\in\Gamma^*}
    f\bigl(\mathfrak s_\Gamma(\omega)+\gamma^*\bigr)\,
    d\mathfrak m(\omega),
    \qquad
    f\in C_c(\mathbb R^d).
\]
The definition is independent of the choice of representatives:
replacing $\mathfrak s_\Gamma(\omega)$ by another representative of the same coset
only reindexes the sum over $\Gamma^*$.

Let $(Z,\nu)$ be a probability-preserving Borel $\Gamma$-space
and put
\[
    (X,\mu)
    :=
    \operatorname{Ind}_\Gamma^{\mathbb R^d}(Z).
\]
The functions that depend only on the coordinate in
$\mathbb R^d/\Gamma$ form the \emph{torus subspace}
\[
    \mathcal K_{\mathrm{tor}}
    :=
    \left\{
        F\in L^2(X,\mu):
        F([t,z])=g(t+\Gamma)
        \text{ for some }
        g\in L^2(\mathbb R^d/\Gamma)
    \right\}.
\]
This is a closed invariant subspace for the Koopman
representation. Set
\[
    \mathcal K_{\mathrm{tor}}^0
    :=
    \mathcal K_{\mathrm{tor}}\cap L^2_0(X,\mu).
\]
We call its orthogonal complement in $L^2_0(X,\mu)$ the
\emph{base subspace}:
\[
    \mathcal K_Z
    :=
    L^2_0(X,\mu)\ominus\mathcal K_{\mathrm{tor}}^0.
\]
Since the Koopman operators are unitary, $\mathcal K_Z$ is also
invariant. Thus
\[
    L^2_0(X,\mu)
    =
    \mathcal K_{\mathrm{tor}}^0
    \oplus
    \mathcal K_Z
\]
is an orthogonal decomposition into invariant subspaces.

\begin{lemma}
\label{lem:induced-spectral-decomposition}
Let $\smash{\mathfrak m_{Z,\nu}^0}$ be a representative of the maximal
spectral type of the $\Gamma$-action on $L^2_0(Z,\nu)$. The
maximal spectral type of the induced $\mathbb R^d$-representation
on $\smash{\mathcal K_{\mathrm{tor}}^0}$ is the counting-measure class on
the set $\smash{\Gamma^*\setminus\{0\}}$. Its maximal spectral type on $\smash{\mathcal K_Z}$ is the periodic lift
$\smash{\widetilde{\mathfrak m}_{Z,\nu}^0}$.

\end{lemma}

\begin{proof}
On $\mathcal K_{\mathrm{tor}}$, the Koopman representation is the
translation representation on $L^2(\mathbb R^d/\Gamma)$.

Its Fourier basis consists of the characters indexed by
$\Gamma^*$. The character indexed by $0$ is constant, so the
maximal spectral type on $\mathcal K_{\mathrm{tor}}^0$ is the
counting-measure class on the set
$\Gamma^*\setminus\{0\}$.

The representation on $\mathcal K_Z$ is naturally equivalent to
the representation induced from the $\Gamma$-representation on
$L^2_0(Z,\nu)$. By the spectral theorem for unitary
representations of abelian groups, it is enough to describe what
induction does to a character
$\omega\in\widehat\Gamma$.

The characters of $\mathbb R^d$ whose restrictions to $\Gamma$
equal $\omega$ are precisely those whose frequencies belong to
$\mathfrak s_\Gamma(\omega)+\Gamma^*$.

Accordingly, the representation induced from $\omega$ is the
direct sum of the characters indexed by this coset, and its
spectral type is counting measure on $\mathfrak s_\Gamma(\omega)+\Gamma^*$.

Applying this description to the spectral type
$\mathfrak m_{Z,\nu}^0$ shows that induction replaces it by the
measure obtained by placing a copy of counting measure on
$\mathfrak s_\Gamma(\omega)+\Gamma^*$ for
$\mathfrak m_{Z,\nu}^0$-almost every $\omega$. By definition, this
is the periodic lift $\widetilde{\mathfrak m}_{Z,\nu}^0$.
\end{proof}

Suppose from now on that the $\Gamma$-action on $(Z,\nu)$ is
weakly mixing. Then $\mathfrak m_{Z,\nu}^0$ is atomless, and the
same is true of its periodic lift $\widetilde{\mathfrak m}_{Z,\nu}^0$.
It is therefore mutually singular with the atomic spectral type
of the torus subspace.

\begin{proposition}
\label{prop:Bartlett-torus-base-decomposition}
Let $(Z,\nu)$ be weakly mixing, and let $Y\subset X$ be a Delone
cross-section. Its Bartlett spectrum has a unique decomposition
\[
    \sigma_Y
    =
    \sigma_Y^{\mathrm{tor}}
    +
    \sigma_Y^Z,
\]
where
\[
    \sigma_Y^{\mathrm{tor}}
    :=
    \sigma_Y|_{\Gamma^*\setminus\{0\}},
    \qquad
    \sigma_Y^Z
    :=
    \sigma_Y-\sigma_Y^{\mathrm{tor}},
\]
and
\[
    \sigma_Y^Z
    \ll
    \widetilde{\mathfrak m}_{Z,\nu}^0.
\]
In particular, $\sigma_Y^Z$ is atomless.
\end{proposition}

\begin{proof}
Choose $g\in\mathcal S(\mathbb R^d)$ such that
\[
    \widehat g(\xi)>0
    \qquad
    (\xi\in\mathbb R^d).
\]
The Delone property implies that the centered linear statistic
$N:=(N_g^Y)^\circ$
belongs to $L^2_0(X,\mu)$. Its spectral measure satisfies
\[
    d\varsigma_N
    =
    |\widehat g|^2\,d\sigma_Y.
\]
Since $\widehat g$ is nowhere zero, the measures $\varsigma_N$
and $\sigma_Y$ are equivalent.

Decompose $N$ according to
\[
    L^2_0(X,\mu)
    =
    \mathcal K_{\mathrm{tor}}^0
    \oplus
    \mathcal K_Z:
\]
\[
    N=N_{\mathrm{tor}}+N_Z.
\]
The two summands lie in orthogonal invariant subspaces. Hence
\[
    \langle U_tN_{\mathrm{tor}},N_Z\rangle
    =
    \langle U_tN_Z,N_{\mathrm{tor}}\rangle
    =
    0
    \qquad
    (t\in\mathbb R^d),
\]
and therefore
\[
    \varsigma_N
    =
    \varsigma_{N_{\mathrm{tor}}}
    +
    \varsigma_{N_Z}.
\]

By Lemma~\ref{lem:induced-spectral-decomposition},
\[
    \operatorname{supp}\varsigma_{N_{\mathrm{tor}}}
    \subset
    \Gamma^*\setminus\{0\},
    \qquad
    \varsigma_{N_Z}
    \ll
    \widetilde{\mathfrak m}_{Z,\nu}^0.
\]
The measure $\widetilde{\mathfrak m}_{Z,\nu}^0$ is atomless, so it
gives no mass to the countable set $\Gamma^*$. It follows that
\[
    \varsigma_{N_{\mathrm{tor}}}
    =
    \varsigma_N|_{\Gamma^*\setminus\{0\}},
\]
while
\[
    \varsigma_{N_Z}
    =
    \varsigma_N-
    \varsigma_N|_{\Gamma^*\setminus\{0\}}
    \ll
    \widetilde{\mathfrak m}_{Z,\nu}^0.
\]

Since $|\widehat g|^2$ is strictly positive, the same decomposition
holds for $\sigma_Y$. Thus
\[
    \sigma_Y^{\mathrm{tor}}
    =
    \sigma_Y|_{\Gamma^*\setminus\{0\}},
    \qquad
    \sigma_Y^Z
    \ll
    \widetilde{\mathfrak m}_{Z,\nu}^0.
\]
The two parts are supported on mutually singular spectral types,
which also proves uniqueness.
\end{proof}

\begin{remark}
The notation
\[
    \sigma_Y^{\mathrm{tor}}
    =
    \sigma_Y|_{\Gamma^*\setminus\{0\}}
\]
does not assert that every point of
$\Gamma^*\setminus\{0\}$ is an atom of $\sigma_Y$.
\end{remark}


\subsection{A nonzero continuous Bartlett component}
\label{subsec:nonzero-continuous-component}

We now show that the base component in
Proposition~\ref{prop:Bartlett-torus-base-decomposition} can be
nonzero.

\begin{theorem}
\label{thm:weakly-mixing-Bartlett-component}
Let $\Gamma\leq\mathbb R^d$ be a lattice and let $(Z,\nu)$ be a
nontrivial weakly mixing probability-preserving Borel
$\Gamma$-space. Then the induced $\mathbb R^d$-space
$\operatorname{Ind}_\Gamma^{\mathbb R^d}(Z)$ admits a generating
stealthy Delone cross-section $Y$ for which
\[
    \sigma_Y
    =
    \sigma_Y^{\mathrm{tor}}
    +
    \sigma_Y^Z,
    \qquad
    \operatorname{supp}\sigma_Y^{\mathrm{tor}}
    \subset
    \Gamma^*\setminus\{0\},
\]
and
\[
    0\neq\sigma_Y^Z
    \ll
    \widetilde{\mathfrak m}_{Z,\nu}^0.
\]
In particular, $\sigma_Y$ has a nonzero continuous component.
\end{theorem}

We first recall the one-dimensional objects used in the
construction. Set
\[
    S(w):=\sin(4\pi w)+\sin(2\pi w),
    \qquad
    \Lambda
    :=
    \frac13\mathbb Z
    \cup
    \left(\frac12+\mathbb Z\right).
\]
Thus $\Lambda$ is the zero set of $S$, and
\[
    \Lambda_0
    :=
    \Lambda\cap[0,1)
    =
    \left\{
        0,\frac13,\frac12,\frac23
    \right\}.
\]
Let $Tz:=e_1.z$, and let $H_z$ be the band-limited coding
function constructed in
Subsection~\ref{subsec:band-limited-orbit-coding}. For
$s\in\mathbb R$, define
\[
    F_{z,s}(w):=S(w)+sH_z(w).
\]
For sufficiently small $|s|$ and every $\lambda\in\Lambda$, we
write $p_{\lambda,s}(z)$ for the unique zero of $F_{z,s}$ near
$\lambda$, and set
\[
    L_{z,s}
    :=
    \{p_{\lambda,s}(z):\lambda\in\Lambda\}.
\]
At $s=0$, one has
\[
    p_{\lambda,0}(z)=\lambda,
    \qquad
    L_{z,0}=\Lambda.
\]

\begin{lemma}
\label{lem:parameter-family}
There is $s_0>0$ such that, for every $|s|<s_0$, the zeros of
$F_{z,s}$ admit an indexing
\[
    L_{z,s}
    =
    \{p_{\lambda,s}(z):\lambda\in\Lambda\}
\]
with the following properties:
\begin{enumerate}
\item
the sets $L_{z,s}$ are real and uniformly Delone, with constants
independent of $z$ and $s$;

\item for all $z \in Z$, 
\[
    L_{Tz,s}=L_{z,s}-1;
\]

\item for all $z \in Z$,
\[
    \widehat{\delta}_{L_{z,s}}
    =
    4\delta_0
    \qquad\text{on }(-1,1);
\]

\item
if $s\neq0$ and
\[
    L_{w,s}=L_{z,s}-u,
\]
then
\[
    u\in\mathbb Z
    \qquad\text{and}\qquad
    w=T^uz;
\]

\item
for every $\lambda\in\Lambda$ and $z\in Z$, the map
$s\mapsto p_{\lambda,s}(z)$ is real analytic and satisfies
\[
    p_{\lambda,0}(z)=\lambda,
    \qquad
    \left.
    \frac{d}{ds}p_{\lambda,s}(z)
    \right|_{s=0}
    =
    -\frac{H_z(\lambda)}{S'(\lambda)}.
\]
\end{enumerate}
\end{lemma}

\begin{proof}
The Rouché estimates in
Subsection~\ref{subsec:real-zero-sets} are uniform when $s$ ranges
over a sufficiently small interval. They give the indexing of the
zeros and the first assertion. The identity
\[
    H_{Tz}(w)=H_z(w+1)
\]
gives the second assertion.

The proof of Proposition~\ref{prop:Fourier-identity} is uniform
over the same interval and gives the third assertion. For
$s\neq0$, the proof of Proposition~\ref{prop:recovery-coding}
also applies without change. Once equality of the corresponding
entire functions has been obtained, division by $s$ recovers the
coding functions $H_z$ and $H_w$. This proves the fourth
assertion.

Finally, every zero is simple. The real-analytic
implicit-function theorem
\cite[Section~2.3]{KP02}, applied to
\[
    S\bigl(p_{\lambda,s}(z)\bigr)
    +
    sH_z\bigl(p_{\lambda,s}(z)\bigr)
    =
    0,
\]
shows that $s\mapsto p_{\lambda,s}(z)$ is real analytic.
Differentiating this identity at $s=0$ gives
\[
    S'(\lambda)
    \left.
    \frac{d}{ds}p_{\lambda,s}(z)
    \right|_{s=0}
    +
    H_z(\lambda)
    =
    0,
\]
which proves the last assertion.
\end{proof}

For $m\in\mathbb Z^{d-1}$, write
$\widetilde m:=(0,m)$ and define
\[
    Q_{z,s}
    :=
    \bigcup_{m\in\mathbb Z^{d-1}}
    \left(
        L_{\widetilde m.z,s}\times\{m\}
    \right).
\]
Let $Y_s$ be the associated cross-section in
\[
    X_{\mathrm{ind}}
    :=
    \operatorname{Ind}_{\mathbb Z^d}^{\mathbb R^d}(Z).
\]
At $s=0$, the family is independent of $z$ and satisfies
\[
    Q_{z,0}
    =
    \Lambda\times\mathbb Z^{d-1}.
\]

Let $P_{\mathrm{tor}}$ be the orthogonal projection of
$L^2(X_{\mathrm{ind}})$ onto the torus subspace
$\mathcal K_{\mathrm{tor}}$. In fundamental-domain coordinates,
it is given by
\[
    (P_{\mathrm{tor}}F)(t)
    =
    \int_ZF([t,z])\,d\nu(z),
    \qquad
    t\in[0,1)^d.
\]

\begin{lemma}
\label{lem:nontorus-variation}
There is $\varphi\in C_c^\infty(\mathbb R^d)$ such that, for
$\Phi_s:=N_\varphi^{Y_s}$, the map
$s\mapsto\Phi_s$ is differentiable at $0$ as an
$L^2(X_{\mathrm{ind}})$-valued map and
\[
    (I-P_{\mathrm{tor}})
    \left.
    \frac{d}{ds}\Phi_s
    \right|_{s=0}
    \neq0.
\]
Consequently,
\[
    (I-P_{\mathrm{tor}})\Phi_s\neq0
\]
for every sufficiently small $s\neq0$.
\end{lemma}

\begin{proof}
Put $D:=[0,1)^d$. Choose $t_0$ in the interior of $D$. Since
$\Lambda\times\mathbb Z^{d-1}$ is uniformly discrete and contains
the origin, there are a neighborhood $E$ of $t_0$ and a function
$\varphi\in C_c^\infty(\mathbb R^d)$ such that, for every
$t\in E$, the only pair $(\lambda,m)\in
\Lambda\times\mathbb Z^{d-1}$ for which
\[
    (\lambda,m)-t\in\operatorname{supp}\varphi
\]
is $(\lambda,m)=(0,0)$, while
$\partial_1\varphi(-t)$ does not vanish identically on $E$.

The compact support of $\varphi$ and the uniform localization of
the zeros in Lemma~\ref{lem:parameter-family} imply that
$\Phi_s([t,z])$ is a sum over only finitely many zero branches,
uniformly for $t\in D$, $z\in Z$, and sufficiently small $s$.
After decreasing $s_0$, the derivatives $\partial_wF_{z,s}$ are
uniformly bounded away from zero at all relevant zeros. Implicit 
differentiation 
therefore gives uniform bounds on
$\partial_s p_{\lambda,s}(z)$. Together with the uniform bound on
the number of contributing branches and the boundedness of
$\nabla\varphi$, this gives a constant $C<\infty$ such that
\[
    \left|
        \frac{\Phi_s([t,z])-\Phi_0([t,z])}{s}
    \right|
    \leq C
\]
for every $0<|s|<s_0$ and almost every $[t,z]\in X_\textrm{ind}$. The
difference quotients converge pointwise to the finite sum below,
so dominated convergence gives convergence in $L^2(X_\textrm{ind})$. This gives
\[
    \left.
    \frac{d}{ds}\Phi_s([t,z])
    \right|_{s=0}
    =
    -
    \sum_{\substack{m\in\mathbb Z^{d-1}\\
                    \lambda\in\Lambda}}
    \frac{H_{\widetilde m.z}(\lambda)}{S'(\lambda)}
    \partial_1\varphi\bigl((\lambda,m)-t\bigr).
\]
For $t\in E$, only $(\lambda,m)=(0,0)$ contributes. Since
$H_z(0)=b(z)$,
\[
    \left.
    \frac{d}{ds}\Phi_s([t,z])
    \right|_{s=0}
    =
    -
    \frac{\partial_1\varphi(-t)}{S'(0)}\,b(z),
    \qquad
    t\in E.
\]

The projection $P_{\mathrm{tor}}$ integrates out the
$z$-variable. Hence
\[
\begin{aligned}
    \left\|
        (I-P_{\mathrm{tor}})
        \left.
        \frac{d}{ds}\Phi_s
        \right|_{s=0}
    \right\|_2^2
    &\geq
    \frac{\operatorname{Var}_\nu(b)}{|S'(0)|^2}
    \int_E
    |\partial_1\varphi(-t)|^2\,dt\\
    &>0.
\end{aligned}
\]
Here $\operatorname{Var}_\nu(b)>0$ because $b$ is injective and
$(Z,\nu)$ is not essentially a one-point space.

Finally, $Q_{z,0}$ is independent of $z$, so
$\Phi_0\in\mathcal K_{\mathrm{tor}}$. Therefore
\[
    (I-P_{\mathrm{tor}})\Phi_s
    =
    s
    (I-P_{\mathrm{tor}})
    \left.
    \frac{d}{ds}\Phi_s
    \right|_{s=0}
    +
    o(s)
\]
in $L^2(X_{\mathrm{ind}})$. The coefficient of $s$ is nonzero, proving the
last assertion.
\end{proof}

\begin{proof}[Proof of
Theorem~\ref{thm:weakly-mixing-Bartlett-component}]
Suppose first that $\Gamma=\mathbb Z^d$. Choose a sufficiently
small $s\neq0$ for which the conclusion of
Lemma~\ref{lem:nontorus-variation} holds.

Lemma~\ref{lem:parameter-family} and the arguments of
Section~\ref{sec:layered-construction} show that $Y_s$ is a
generating stealthy Delone cross-section. Write
\[
    \sigma_{Y_s}
    =
    \sigma_{Y_s}^{\mathrm{tor}}
    +
    \sigma_{Y_s}^Z
\]
for the decomposition supplied by
Proposition~\ref{prop:Bartlett-torus-base-decomposition}.

Since the constants belong to the torus subspace,
\[
    (I-P_{\mathrm{tor}})\Phi_s^\circ
    =
    (I-P_{\mathrm{tor}})\Phi_s
    \neq0.
\]
The spectral measure of this vector is therefore nonzero.
Moreover,
\[
    d\varsigma_{\Phi_s^\circ}
    =
    |\widehat\varphi|^2\,d\sigma_{Y_s}.
\]
The torus and base subspaces are orthogonal and invariant, and
their maximal spectral types are mutually singular. Projecting
onto the base subspace therefore gives
\[
    d\varsigma_{(I-P_{\mathrm{tor}})\Phi_s^\circ}
    =
    |\widehat\varphi|^2\,d\sigma_{Y_s}^Z.
\]
The measure on the left is nonzero, and hence
\[
    \sigma_{Y_s}^Z\neq0.
\]

Now let $\Gamma$ be arbitrary. Choose a linear isomorphism
$A:\mathbb R^d\to\mathbb R^d$ such that
$A\mathbb Z^d=\Gamma$, and regard $Z$ as a
$\mathbb Z^d$-space by setting
\[
    n.z:=(An).z.
\]
This action is again weakly mixing. Apply the preceding
construction and set
\[
    P_{z,s}:=AQ_{z,s}.
\]
The associated cross-section in
$\operatorname{Ind}_\Gamma^{\mathbb R^d}(Z)$ is generating,
stealthy, and Delone by the proof of
Theorem~\ref{thm:lattice-coding}.

The isomorphism induced by $A$ maps the torus subspace onto the
torus subspace and its orthogonal complement onto the base
subspace. If $\varphi$ is supplied by
Lemma~\ref{lem:nontorus-variation}, put
\[
    \varphi_A(x):=\varphi(A^{-1}x).
\]
The corresponding linear statistic for the transformed family
has a nonzero component in the base subspace. The preceding
spectral argument therefore gives
\[
    0\neq\sigma_Y^Z
    \ll
    \widetilde{\mathfrak m}_{Z,\nu}^0.
\]
Since the lifted maximal spectral type is atomless, $\sigma_Y$
is not pure point.
\end{proof}

\begin{proposition}[Persistence of a torus atom]
\label{prop:persistence-torus-atom}
For the cross-sections $Y_s \subset X_{\textrm{ind}}$ constructed above when
$\Gamma=\mathbb Z^d$,
\[
    \sigma_{Y_s}(\{e_1\})>0
\]
for every sufficiently small $s\neq0$. If
$\Gamma=A\mathbb Z^d$, the corresponding transformed
cross-section has a Bartlett atom at
$A^{-\mathsf T}e_1$.
\end{proposition}

\begin{proof}
Recall that
\[
    \Lambda_0
    :=
    \Lambda\cap[0,1)
    =
    \left\{
        0,\frac13,\frac12,\frac23
    \right\}.
\]
Define
\[
    c(s)
    :=
    \int_Z
    \sum_{\lambda\in\Lambda_0}
    e^{-2\pi i p_{\lambda,s}(z)}
    \,d\nu(z).
\]

Let $\varphi\in\mathcal S(\mathbb R^d)$. The equivariance of the
zero branches gives
\[
    p_{\lambda+n,s}(z)
    =
    n+p_{\lambda,s}(T^nz),
    \qquad
    \lambda\in\Lambda_0,\quad n\in\mathbb Z.
\]
Using this identity and the invariance of $\nu$, we obtain, for
$t\in[0,1)^d$,
\[
\begin{aligned}
    (P_{\mathrm{tor}}N_\varphi^{Y_s})(t)
    &=
    \sum_{\lambda\in\Lambda_0}
    \int_Z
    \sum_{(n,m)\in\mathbb Z\times\mathbb Z^{d-1}}
    \varphi\bigl(
        (n+p_{\lambda,s}(z),m)-t
    \bigr)
    \,d\nu(z).
\end{aligned}
\]
For $k\in\mathbb Z^d$, write
\[
    E_k(t):=e^{-2\pi i\langle k,t\rangle}.
\]
A change of variables gives
\[
\begin{aligned}
    \left\langle
        P_{\mathrm{tor}}N_\varphi^{Y_s},
        E_k
    \right\rangle
    &=
    \widehat\varphi(k)
    \int_Z
    \sum_{\lambda\in\Lambda_0}
    e^{2\pi i k_1p_{\lambda,s}(z)}
    \,d\nu(z).
\end{aligned}
\]
In particular,
\[
    \left\langle
        P_{\mathrm{tor}}N_\varphi^{Y_s},
        E_{e_1}
    \right\rangle
    =
    \widehat\varphi(e_1)\,\overline{c(s)}.
\]

At $s=0$,
\[
\begin{aligned}
    c(0)
    &=
    \sum_{\lambda\in\Lambda_0}
    e^{-2\pi i\lambda}\\
    &=-1.
\end{aligned}
\]
The four zero branches depend continuously on $s$, so $c(s)$
remains nonzero for every sufficiently small $s$.

Choose $\varphi$ such that $\widehat\varphi(e_1)\neq0$. Since
$E_{e_1}$ belongs to the torus subspace,
\[
    \left\langle
        N_\varphi^{Y_s},
        E_{e_1}
    \right\rangle
    =
    \left\langle
        P_{\mathrm{tor}}N_\varphi^{Y_s},
        E_{e_1}
    \right\rangle
    \neq0.
\]
Moreover, $E_{e_1}$ is orthogonal to the constants, so centering
does not change this coefficient. Hence the spectral measure of
$(N_\varphi^{Y_s})^\circ$ has a nonzero atom at $e_1$. By the
Bartlett--Koopman identity,
\[
    d\varsigma_{(N_\varphi^{Y_s})^\circ}
    =
    |\widehat\varphi|^2\,d\sigma_{Y_s},
\]
and therefore
\[
    \sigma_{Y_s}(\{e_1\})>0.
\]

Under the linear change of variables
$\Gamma=A\mathbb Z^d$, the frequency $e_1$ is carried to
$A^{-\mathsf T}e_1$, which proves the final assertion.
\end{proof}

\begin{remark}[Finite local complexity]
\label{rem:no-finite-local-complexity}
In dimension one, the nonperiodic realizations produced above do
not have finite local complexity. Indeed, choose a nonzero
$f\in\mathcal S(\mathbb R)$ such that
\[
    \operatorname{supp}\widehat f\subset(-1,1)
    \qquad\text{and}\qquad
    \widehat f(0)\neq0.
\]
The realization-level Fourier identity gives
\[
    f*\delta_{L_{z,s}}
    =
    4\widehat f(0).
\]
Thus $f$ tiles the line at a nonzero level with translation set
$L_{z,s}$. If $L_{z,s}$ had finite local complexity, the
periodicity theorem of Kolountzakis and Lev
\cite{KL16} would imply that $L_{z,s}$ is periodic.

If $L_{z,s}=L_{z,s}-u$ for some $u\neq0$, the orbit-recovery
property gives
\[
    u\in\mathbb Z
    \qquad\text{and}\qquad
    T^uz=z.
\]
For a nontrivial weakly mixing base, this occurs only on a null
set. Hence almost every one-dimensional realization has infinite
local complexity.
\end{remark}

\begin{question}[Cloaking]
\label{que:torus-cloaking}
Let $\Gamma\leq\mathbb R^d$ be a lattice and let $(Z,\nu)$ be a
nontrivial weakly mixing probability-preserving Borel
$\Gamma$-space. Does
$\smash{\operatorname{Ind}_{\Gamma}^{\mathbb R^d}(Z)}$ admit a generating
stealthy Delone cross-section $Y$ such that
\[
    \sigma_Y(\Gamma^*\setminus\{0\})=0?
\]
Equivalently, can the eigenvalues arising from the torus factor
be invisible to all centered linear statistics?
\end{question}

In diffraction terminology, the question asks whether all
nonzero torus eigenvalues can be extinctions for the chosen
cross-section.


\subsection{Bernoulli bases and singular spectral types}
\label{subsec:examples-continuous-spectral-types}

We first compare the preceding construction with two standard
point processes obtained from Bernoulli lattice systems. Let
\[
    \rho_\Gamma
    :=
    \operatorname{covol}(\Gamma)^{-1},
    \qquad
    (\gamma.z)_\alpha:=z_{\alpha+\gamma}
    \quad
    (\alpha,\gamma\in\Gamma).
\]

Consider first the Bernoulli space
$Z=\{0,1\}^{\Gamma}$ with
\[
    \nu\{z:z_0=1\}=p,
    \qquad
    0<p<1,
\]
and define the thinned lattice
\[
    P_z^{\mathrm{thin}}
    :=
    \{\gamma\in\Gamma:z_\gamma=1\}.
\]
After restriction to an invariant conull set, this equivariant
family gives a generating separated cross-section: the set
$P_z^{\mathrm{thin}}$ determines the Bernoulli configuration
$z$. It is almost surely not relatively dense, since the
Bernoulli field contains empty lattice boxes of arbitrarily large
diameter.

The independence of the Bernoulli variables gives the diagonal
covariance term, while Poisson summation gives the contribution
from the mean lattice. Thus
\[
    \sigma_{\mathrm{thin}}
    =
    \rho_\Gamma p(1-p)\lambda_d
    +
    \rho_\Gamma^2p^2
    \sum_{\kappa\in\Gamma^*\setminus\{0\}}
    \delta_\kappa.
\]
In particular, Bernoulli thinning is neither hyperuniform nor
stealthy.

For independent lattice displacements, let $\vartheta$ be a
probability measure supported on a bounded Borel neighborhood
$K\subset\mathbb R^d$ of the origin, chosen sufficiently small
that the translates $\gamma+K$, $\gamma\in\Gamma$, are uniformly
separated. On the Bernoulli space $Z=K^\Gamma$, define
\[
    P_z^{\mathrm{disp}}
    :=
    \{\gamma+z_\gamma:\gamma\in\Gamma\}.
\]
After restriction to an invariant conull set, the associated
cross-section is generating and uniformly Delone. Indeed, each
point can be assigned to its unique lattice cell, and the
resulting displacement field determines $z$.

Let
\[
    \phi(\xi)
    :=
    \int_{\mathbb R^d}
    e^{-2\pi i\langle u,\xi\rangle}\,d\vartheta(u)
\]
be the characteristic function of the displacement distribution.
The Bartlett spectrum is
\[
    \sigma_{\mathrm{disp}}
    =
    \rho_\Gamma
    \bigl(1-|\phi(\xi)|^2\bigr)\,d\lambda_d(\xi)
    +
    \rho_\Gamma^2
    \sum_{\kappa\in\Gamma^*\setminus\{0\}}
    |\phi(\kappa)|^2\delta_\kappa.
\]
Since $\phi(0)=1$ and $\phi$ is continuous,
\[
    1-|\phi(\xi)|^2\longrightarrow0
    \qquad
    (\xi\longrightarrow0),
\]
so the process is hyperuniform. It is not stealthy unless
$\vartheta$ is supported on a single point. Indeed, if
$|\phi|=1$ on a neighborhood of the origin, then
\[
    e^{-2\pi i\langle u-v,\xi\rangle}=1
\]
for $\vartheta\otimes\vartheta$-almost every $(u,v)$ and every
$\xi$ in that neighborhood. Varying $\xi$ then forces $u=v$
almost surely.

Dependent perturbations can behave differently. Recent
stationary-increment perturbations erase Bragg peaks and yield
hyperuniform point processes in dimension one
\cite{TLS26}; the constructions there do not produce an exact
open spectral gap.

These examples indicate what is different about our construction.
The standard lattice cross-section is Delone and stealthy but
forgets the Bernoulli base. Thinning is generating but not
relatively dense, while independent displacement is generating
and Delone but has no open spectral gap. By contrast, already in
dimension one our construction depends on the full orbit through
\[
    H_z(w)
    =
    \sum_{n\in\mathbb Z}
    b(T^nz)\psi(w-n).
\]
This nonlocal dependence allows generation, uniform Delone
geometry, and an exact open Fourier gap to hold simultaneously.

For Bernoulli bases, the continuous component obtained in
Theorem~\ref{thm:weakly-mixing-Bartlett-component} is absolutely
continuous with respect to Lebesgue measure.

\begin{corollary}
\label{cor:Bernoulli-Bartlett-component}
Let $\Gamma\leq\mathbb R^d$ be a lattice and let $(Z,\nu)$ be a
nontrivial Bernoulli $\Gamma$-space. Then the induced space
$\smash{\operatorname{Ind}_\Gamma^{\mathbb R^d}(Z)}$ admits a generating
stealthy Delone cross-section $Y$ whose Bartlett spectrum has the
form
\[
    \sigma_Y
    =
    \sigma_Y^{\mathrm{tor}}
    +
    h\,d\lambda_d,
    \qquad
    \operatorname{supp}\sigma_Y^{\mathrm{tor}}
    \subset\Gamma^*\setminus\{0\},
\]
where $h\in L^1_{\mathrm{loc}}(\mathbb R^d)$ is nonnegative and
not identically zero. Moreover, $h$ vanishes almost everywhere on
a neighborhood of the origin.
\end{corollary}

\begin{proof}
Choose an orthonormal basis of the one-coordinate $L^2$-space
which contains the constant function. The tensors involving only
finitely many nonconstant coordinates form an orthonormal basis
of the Bernoulli space. Translation by $\Gamma$ permutes these
tensors.

Every nonconstant finite-support tensor has a free
$\Gamma$-orbit. Indeed, a nonzero element of the torsion-free
group $\Gamma$ cannot preserve a nonempty finite set of
coordinates. Each such orbit therefore spans a copy of the
regular representation of $\Gamma$, and the Koopman
representation on $L^2_0(Z,\nu)$ is a direct sum of copies of the
regular representation.

Its maximal spectral type is Haar measure on
$\widehat\Gamma$. The periodic lift of Haar measure is equivalent
to $\lambda_d$, so
Theorem~\ref{thm:weakly-mixing-Bartlett-component} gives
\[
    0\neq\sigma_Y^Z\ll\lambda_d.
\]
Writing $\sigma_Y^Z=h\,d\lambda_d$ proves the asserted
decomposition. Since the whole Bartlett spectrum vanishes on a
neighborhood of the origin, the same is true of $h$ there.
\end{proof}

The same theorem produces singular-continuous Bartlett components
when the base has singular maximal spectral type.

\begin{corollary}
\label{cor:singular-continuous-Bartlett-component}
Under the assumptions of
Theorem~\ref{thm:weakly-mixing-Bartlett-component}, suppose that
$\smash{\mathfrak m_{Z,\nu}^0}$ is singular with respect to Haar measure
on $\smash{\widehat\Gamma}$. Then the cross-section may be chosen so that
$\smash{\sigma_Y^Z}$ is a nonzero singular-continuous measure on
$\mathbb R^d$.
\end{corollary}

\begin{proof}
The periodic lift
$\widetilde{\mathfrak m}_{Z,\nu}^0$ is singular with respect to
$\lambda_d$. It is also atomless because the base action is
weakly mixing. Since
\[
    0\neq\sigma_Y^Z
    \ll
    \widetilde{\mathfrak m}_{Z,\nu}^0,
\]
the measure $\sigma_Y^Z$ is nonzero, singular, and atomless.
\end{proof}

\begin{example}[Gaussian bases]
Nontrivial weakly mixing $\Gamma$-actions with singular maximal
spectral type exist for every lattice $\Gamma$. Choose an
identification
\[
    \Gamma\cong\mathbb Z^d,
    \qquad
    \widehat\Gamma\cong\mathbb T^d.
\]
It is enough to construct a symmetric continuous probability
measure $\theta$ on $\mathbb T^d$ such that every convolution
power $\theta^{*n}$ is singular with respect to Haar measure.

Let $E\subset\mathbb T$ be the image of
\[
    \left\{
        \sum_{k\geq1}\varepsilon_k2^{-k!}
        :
        \varepsilon_k\in\{0,1\}
    \right\}.
\]
The rapid separation of the scales $2^{-k!}$ shows that $E$ is
compact and perfect and has upper box dimension zero. Indeed, if
\[
    2^{-(N+1)!}\leq r<2^{-N!},
\]
then fixing the first $N+1$ digits covers $E$ by at most
$2^{N+1}$ intervals of length $2r$. Hence
\[
    \overline{\dim}_{\mathrm B}E
    \leq
    \limsup_{N\to\infty}
    \frac{(N+1)\log 2}{N!\log 2}
    =
    0.
\]
Choose a
continuous probability measure $\nu_0$ supported on $E$, let
\[
    \iota:\mathbb T\longrightarrow\mathbb T^d,
    \qquad
    \iota(x):=(x,0,\ldots,0),
\]
and define
\[
    \theta
    :=
    \frac12\iota_*\nu_0
    +
    \frac12(-\iota)_*\nu_0.
\]
Then $\theta$ is symmetric and atomless.

Upper box dimension is subadditive under finite Minkowski sums.
Consequently, for every $n\geq1$, the support of
$\theta^{*n}$ has upper box dimension zero and hence Haar measure
zero. Thus
\[
    \theta^{*n}\perp m_{\mathbb T^d}
    \qquad
    (n\geq1).
\]
Since convolution with an atomless measure cannot create atoms,
every $\theta^{*n}$ is also atomless.

Let $(Z,\nu)$ be the Gaussian $\Gamma$-system whose first chaos
has spectral measure $\theta$. Its maximal spectral type on
$L^2_0(Z,\nu)$ is represented by
\[
    \sum_{n=1}^{\infty}
    \frac{1}{n!}\theta^{*n},
\]
and the action is weakly mixing because $\theta$ is atomless
\cite[Section~6.4.1]{KT06}. The displayed measure is atomless
and is supported on a countable union of Haar-null sets; it is
therefore singular with respect to Haar measure. Hence
Corollary~\ref{cor:singular-continuous-Bartlett-component}
produces a generating stealthy Delone cross-section with a
nonzero singular-continuous Bartlett component.
\end{example}

\begin{question}
\label{que:Bernoulli-density}
Can the Bernoulli construction be chosen so that its absolutely
continuous density satisfies
\[
    h(\xi)>0
    \qquad
    \text{for almost every $\xi$ outside the spectral gap?}
\]
\end{question}


\section{Density bounds for spectral gaps}
\label{sec:density-bounds}

A spectral gap imposes a lower bound on the intensity of a
uniformly separated point process. The vanishing of the Bartlett
spectrum first turns almost every configuration into an exact
sampling set for suitable Paley--Wiener spaces. Landau's necessary
density condition then gives the lower bound.

For an open neighborhood $\Omega$ of the origin, set
\[
    \mathcal D(\Omega)
    :=
    \sup\left\{
        \lambda_d(K):
        K\subset\mathbb R^d\ \text{compact and }K-K\subset\Omega
    \right\}.
\]

\begin{theorem}[Density bound]
\label{thm:density-bound}
Let $\eta$ be a translation-invariant point process of
intensity $\rho$. Suppose that, for some $r>0$, the process is
supported on $r$-separated configurations. Then
\[
    \sigma_\eta(\Omega)=0
    \qquad\Longrightarrow\qquad
    \rho\geq\mathcal D(\Omega)
\]
for every open neighborhood $\Omega$ of the origin. In particular,
if $C\subset\Omega$ is a symmetric convex body, then
\[
    \rho\geq2^{-d}\lambda_d(C).
\]
\end{theorem}

\subsection{From a spectral gap to exact sampling}
\label{subsec:gap-to-sampling}

For a compact set $K\subset\mathbb R^d$, the associated
Paley--Wiener space is
\[
    PW_K
    :=
    \left\{
        f\in L^2(\mathbb R^d):
        \operatorname{supp}\widehat f\subset K
    \right\}.
\]
Since $\widehat f\in L^2(K)\subset L^1(K)$, every $f\in PW_K$
has a continuous representative, so its values on a discrete set
are well defined.

A set $P\subset\mathbb R^d$ is called a \emph{sampling set} for
$PW_K$ if there are constants $0<A\leq B<\infty$ such that
\[
    A\|f\|_2^2
    \leq
    \sum_{p\in P}|f(p)|^2
    \leq
    B\|f\|_2^2
    \qquad
    (f\in PW_K).
\]
Thus the samples $(f(p))_{p\in P}$ determine $f$ and do so
stably: the lower bound prevents two different functions from
having nearly identical samples, while the upper bound ensures
that the sampling map
\[
    f\longmapsto (f(p))_{p\in P}
\]
is bounded from $PW_K$ to $\ell^2(P)$. This is stronger than
merely requiring $P$ to be a uniqueness set, which only says that
$f|_P=0$ implies $f=0$. We refer to
\cite[Sections~2.3, 2.5, and~3.8]{OU16} for background on
Paley--Wiener spaces and stable sampling.

For uniformly separated sets, the upper sampling bound follows
from the Plancherel--P\'olya estimate below. The substantive issue
is therefore the lower bound. In our setting the spectral gap
gives considerably more: for almost every realization,
\[
    \sum_{p\in P}|f(p)|^2
    =
    \rho\|f\|_2^2
    \qquad
    (f\in PW_K).
\]
We call this \emph{exact sampling}. It says that the sampling map,
after multiplication by $\rho^{-1/2}$, is an isometry.

Equivalently, after Fourier transformation, the exponentials
\[
    \left\{
        \xi\longmapsto e^{-2\pi i\langle p,\xi\rangle}
        :
        p\in P
    \right\}
\]
form a tight frame for $L^2(K)$ with frame bound $\rho$.

\begin{lemma}[Plancherel--P\'olya estimate]
\label{lem:Plancherel-Polya}
Let $K\subset\mathbb R^d$ be compact and let $r>0$. There is a
constant $C=C(K,r)$ such that
\[
    \sum_{p\in P}|f(p)|^2
    \leq
    C\|f\|_2^2
\]
for every $r$-separated set $P\subset\mathbb R^d$ and every
$f\in PW_K$.
\end{lemma}

\begin{proof}
Choose an integer $m>d/2$. The local Sobolev inequality gives
\[
    |f(p)|^2
    \leq
    C_{d,m,r}
    \sum_{|\alpha|\leq m}
    \int_{B_{r/2}(p)}
    |\partial^\alpha f(x)|^2\,d\lambda_d(x).
\]
Since the balls $B_{r/2}(p)$ are pairwise disjoint,
\[
    \sum_{p\in P}|f(p)|^2
    \leq
    C_{d,m,r}
    \sum_{|\alpha|\leq m}
    \|\partial^\alpha f\|_2^2.
\]
Plancherel's theorem and
$\operatorname{supp}\widehat f\subset K$ give
\[
    \|\partial^\alpha f\|_2
    \leq
    \left(
        2\pi\sup_{\xi\in K}\|\xi\|
    \right)^{|\alpha|}
    \|f\|_2,
\]
and the assertion follows.
\end{proof}

\begin{proposition}
\label{prop:exact-sampling}
Under the assumptions of
Theorem~\ref{thm:density-bound}, let
$K\subset\mathbb R^d$ be compact and satisfy
$K-K\subset\Omega$ and $\lambda_d(\partial K)=0$. Then, for
$\eta$-almost every configuration $P$,
\[
    \sum_{p\in P}|f(p)|^2
    =
    \rho\int_{\mathbb R^d}|f(x)|^2\,d\lambda_d(x)
    \qquad
    (f\in PW_K).
\]
In particular, almost every $P$ is a sampling set for $PW_K$.
\end{proposition}

\begin{proof}
First let $f\in PW_K$ satisfy
$\widehat f\in C_c^\infty(\operatorname{int}K)$. Then
$f\in\mathcal S(\mathbb R^d)$ and since
\[
    \widehat{|f|^2}(\xi)
    =
    \int_{\mathbb R^d}
    \widehat f(\eta)
    \overline{\widehat f(\eta-\xi)}
    \,d\eta,
\]
one has
\[
    \operatorname{supp}\widehat{|f|^2}
    \subset
    K-K
    \subset
    \Omega.
\]
The Bartlett identity therefore gives
\[
\begin{aligned}
    \operatorname{Var}_\eta
    \left(
        \sum_{p\in P}|f(p)|^2
    \right)
    &=
    \int_{\mathbb R^d}
    \left|
        \widehat{|f|^2}(\xi)
    \right|^2
    \,d\sigma_\eta(\xi)\\
    &=0.
\end{aligned}
\]
The random variable is almost surely equal to its expectation, and
the intensity formula yields
\[
    \sum_{p\in P}|f(p)|^2
    =
    \mathbb E_\eta
    \left[
        \sum_{p\in P}|f(p)|^2
    \right]
    =
    \rho\|f\|_2^2
\]
for almost every $P$.

Since $\lambda_d(\partial K)=0$, the space
$C_c^\infty(\operatorname{int}K)$ is dense in $L^2(K)$. By
Plancherel's theorem, we may choose a countable dense family
$(f_n)_{n\geq1}$ in $PW_K$ such that
$\widehat f_n\in C_c^\infty(\operatorname{int}K)$. After
intersecting the corresponding conull sets, the identity holds
simultaneously for every $f_n$.

Fix a configuration $P$ in this conull set. By
Lemma~\ref{lem:Plancherel-Polya},
\[
    \sum_{p\in P}|f(p)|^2
    \leq
    C(K,r)\|f\|_2^2
    \qquad
    (f\in PW_K).
\]
Thus both sides of the sampling identity depend continuously on
$f\in PW_K$. Density of $(f_n)$ extends the identity to every
$f\in PW_K$.
\end{proof}

\begin{remark}
\label{rem:realization-Fourier-identity}
The same argument shows that, after discarding a null set of
configurations,
\[
    \widehat{\delta}_P
    =
    \rho\delta_0
    \qquad\text{on }\Omega
\]
in the sense of tempered distributions. Indeed, apply the zero-variance argument 
to a countable dense
family in $C_c^\infty(\Omega)$, intersect the resulting conull
sets, and use the uniform temperedness of the counting measures
of $r$-separated configurations. Proposition~\ref{prop:exact-sampling} is the corresponding
quadratic identity, obtained by testing against $|f|^2$.
\end{remark}


\subsection{Landau's density theorem}
\label{subsec:Landau-density}

For a uniformly discrete set $P\subset\mathbb R^d$, its lower
Beurling density is
\[
    D^-(P)
    :=
    \liminf_{R\to\infty}
    \inf_{x\in\mathbb R^d}
    \frac{\#(P\cap B_R(x))}
         {\lambda_d(B_R(0))}.
\]

We use Landau's necessary density theorem
\cite{Lan67}; see also \cite[Theorem~5.28]{OU16}. In our
Fourier-transform normalization, it takes the following form.

\begin{theorem}[Landau]
\label{thm:Landau}
Let $K\subset\mathbb R^d$ be compact with
$\lambda_d(\partial K)=0$. If a uniformly discrete set
$P\subset\mathbb R^d$ is a sampling set for $PW_K$, then
\[
    D^-(P)\geq\lambda_d(K).
\]
\end{theorem}

Olevskii and Ulanovskii use the Fourier phase
$e^{-i\langle x,\xi\rangle}$. Thus the spectral set corresponding
to $K$ in their normalization is $2\pi K$, and their bound becomes
\[
    (2\pi)^{-d}\lambda_d(2\pi K)
    =
    \lambda_d(K),
\]
which is the formulation above.

\begin{proof}[Proof of Theorem~\ref{thm:density-bound}]
We first reduce to the ergodic case. Let
\[
    \eta
    =
    \int_A\eta_\alpha\,d\pi(\alpha)
\]
be the ergodic decomposition of $\eta$. After discarding a
$\pi$-null set, each $\eta_\alpha$ has local second moments and
is supported on $r$-separated configurations. Write
$\rho_\alpha$ and $\sigma_\alpha$ for its intensity and
Bartlett spectrum.

For $f\in\mathcal S(\mathbb R^d)$, the variance decomposition
gives
\[
\begin{aligned}
    \operatorname{Var}_\eta(N_f)
    &=
    \int_A
    \operatorname{Var}_{\eta_\alpha}(N_f)
    \,d\pi(\alpha)\\
    &\quad+
    \operatorname{Var}_\pi(\rho_\alpha)
    \left|
        \int_{\mathbb R^d}f\,d\lambda_d
    \right|^2.
\end{aligned}
\]
By uniqueness of the Bartlett spectrum,
\[
    \sigma_\eta
    =
    \int_A\sigma_\alpha\,d\pi(\alpha)
    +
    \operatorname{Var}_\pi(\rho_\alpha)\delta_0.
\]
Since $\Omega$ is a neighborhood of the origin and
$\sigma_\eta(\Omega)=0$, positivity gives
\[
    \operatorname{Var}_\pi(\rho_\alpha)=0,
    \qquad
    \sigma_\alpha(\Omega)=0
\]
for $\pi$-almost every $\alpha$. Hence
\[
    \rho_\alpha=\rho
\]
for $\pi$-almost every $\alpha$.

It therefore suffices to prove the assertion for each ergodic
component. We may assume from now on that $\eta$ is ergodic.

First let $K\subset\mathbb R^d$ be compact, satisfy
$\lambda_d(\partial K)=0$, and have $K-K\subset\Omega$. By
Proposition~\ref{prop:exact-sampling}, almost every realization
$P$ is a sampling set for $PW_K$. Hence
\[
    D^-(P)\geq\lambda_d(K)
\]
by Theorem~\ref{thm:Landau}.

Choose a nonnegative function
$\varphi\in C_c(\mathbb R^d)$ with integral one and
$\operatorname{supp}\varphi\subset B_C(0)$. Wiener's pointwise
ergodic theorem for $\mathbb R^d$-actions \cite{Wie39} gives, for
almost every $P$,
\[
    \frac{1}{\lambda_d(B_R(0))}
    \int_{B_R(0)}
    N_\varphi(t.P)\,d\lambda_d(t)
    \longrightarrow\rho.
\]
For $R>C$, positivity of $\varphi$ and its support condition imply
\[
    \int_{B_{R-C}(0)}N_\varphi(t.P)\,d\lambda_d(t)
    \leq
    \#(P\cap B_R(0))
    \leq
    \int_{B_{R+C}(0)}N_\varphi(t.P)\,d\lambda_d(t).
\]
Dividing by $\lambda_d(B_R(0))$ and applying the ergodic theorem
shows that
\[
    \lim_{R\to\infty}
    \frac{\#(P\cap B_R(0))}
         {\lambda_d(B_R(0))}
    =
    \rho.
\]
Consequently,
\[
    \lambda_d(K)
    \leq
    D^-(P)
    \leq
    \liminf_{R\to\infty}
    \frac{\#(P\cap B_R(0))}
         {\lambda_d(B_R(0))}
    =
    \rho.
\]

We now remove the assumption
$\lambda_d(\partial K)=0$. Let $K\subset\mathbb R^d$ be compact
with $K-K\subset\Omega$. Since $K-K$ is compact and $\Omega$ is
open, there is $r_0>0$ such that
\[
    (K-K)+\overline{B_{2r_0}(0)}
    \subset\Omega.
\]
For $0<r<r_0$, set
\[
    K_r:=K+\overline{B_r(0)}.
\]
The coarea theorem applied to the distance function
$x\mapsto\operatorname{dist}(x,K)$ shows that
$\lambda_d(\partial K_r)=0$ for almost every $r$. Moreover,
\[
    K_r-K_r
    \subset
    (K-K)+\overline{B_{2r}(0)}
    \subset\Omega.
\]
For each such $r$, the preceding argument gives
\[
    \rho\geq\lambda_d(K_r)\geq\lambda_d(K).
\]
Taking the supremum over all compact $K$ with
$K-K\subset\Omega$ proves
\[
    \rho\geq\mathcal D(\Omega).
\]

Finally, if $C\subset\Omega$ is a symmetric convex body, then
$K:=\frac12C$ satisfies
\[
    K-K=C,
    \qquad
    \lambda_d(K)=2^{-d}\lambda_d(C).
\]
Thus
\[
    \rho\geq2^{-d}\lambda_d(C),
\]
as required.
\end{proof}

\begin{corollary}
\label{cor:ball-gap-density}
Under the assumptions of
Theorem~\ref{thm:density-bound}, if
$\sigma_\eta(B_r(0))=0$, then
\[
    \rho
    \geq
    \lambda_d\bigl(B_{r/2}(0)\bigr).
\]
In dimension one,
\[
    \sigma_\eta((-r,r))=0
    \qquad\Longrightarrow\qquad
    \rho\geq r.
\]
\end{corollary}

\begin{corollary}
\label{cor:convex-gap-density}
Let $\Omega\subset\mathbb R^d$ be a symmetric open convex set
of finite Lebesgue measure. Then
\[
    D(\Omega)=2^{-d}\lambda_d(\Omega).
\]
Consequently, under the assumptions of
Theorem~\ref{thm:density-bound},
\[
    \sigma_\eta(\Omega)=0
    \quad\Longrightarrow\quad
    \rho_\eta\geq2^{-d}\lambda_d(\Omega).
\]
\end{corollary}

\begin{proof}
Let $K\subset\mathbb R^d$ be compact and satisfy
$K-K\subset\Omega$. By the Brunn--Minkowski inequality,
\[
    \lambda_d(K-K)^{1/d}
    \geq
    \lambda_d(K)^{1/d}
    +
    \lambda_d(-K)^{1/d}
    =
    2\lambda_d(K)^{1/d}.
\]
Hence
\[
    \lambda_d(K)
    \leq
    2^{-d}\lambda_d(K-K)
    \leq
    2^{-d}\lambda_d(\Omega).
\]
Thus
\[
    D(\Omega)\leq2^{-d}\lambda_d(\Omega).
\]

Conversely, for $0<r<1$, the compact set
\[
    K_r:=\frac r2\,\overline{\Omega}
\]
satisfies
\[
    K_r-K_r=r\overline{\Omega}\subset\Omega,
\]
and
\[
    \lambda_d(K_r)
    =
    2^{-d}r^d\lambda_d(\Omega).
\]
Letting $r\uparrow1$ gives the reverse inequality.
\end{proof}

\begin{remark}[Sharpness for lattice processes]
\label{rem:Blichfeldt-sharpness}
Let $\Gamma\leq\mathbb R^d$ be a lattice, and let $\eta_\Gamma$
be the point process obtained by translating $\Gamma$ uniformly
over a fundamental domain. Writing
$\rho_\Gamma:=\operatorname{covol}(\Gamma)^{-1}$ for its
intensity, its Bartlett spectrum is
\[
    \sigma_{\eta_\Gamma}
    =
    \rho_\Gamma^2
    \sum_{\kappa\in\Gamma^*\setminus\{0\}}
    \delta_\kappa.
\]
Hence the largest open set on which its Bartlett spectrum
vanishes is
\[
    \Omega_\Gamma
    :=
    \mathbb R^d\setminus
    \bigl(\Gamma^*\setminus\{0\}\bigr).
\]

Suppose that $K\subset\mathbb R^d$ is compact and
$K-K\subset\Omega_\Gamma$. Then no two distinct points of $K$
differ by an element of $\Gamma^*$. Blichfeldt's theorem
\cite[Chapter~III, \S2, Theorem~I]{Cas97} therefore gives
\[
    \lambda_d(K)
    \leq
    \operatorname{covol}(\Gamma^*)
    =
    \rho_\Gamma.
\]
Conversely, compact sets exhausting the interior of a fundamental
parallelepiped for $\Gamma^*$ have difference sets contained in
$\Omega_\Gamma$ and volumes converging to
$\operatorname{covol}(\Gamma^*)$. Consequently,
\[
    \mathcal D(\Omega_\Gamma)
    =
    \operatorname{covol}(\Gamma^*)
    =
    \rho_\Gamma.
\]
Thus Theorem~\ref{thm:density-bound} is sharp for lattice
processes.
\end{remark}

For the gap
\[
    \Omega_A
    =
    A^{-\mathsf T}(-1,1)^d
\]
in Theorem~\ref{thm:lattice-coding}, the
preceding corollary gives
\[
    D(\Omega_A)
    =
    2^{-d}\lambda_d(\Omega_A)
    =
    \frac{1}{|\det A|}
    =
    \frac{1}{\operatorname{covol}(\Gamma)}.
\]
The intensity of the constructed cross-section is
\[
    \rho_Y
    =
    \frac{4}{\operatorname{covol}(\Gamma)}
    =
    4D(\Omega_A).
\]
Thus the present construction has intensity four times the
optimal lower bound associated with its gap.


\section{Positive stealthy random measures}
\label{sec:positive-random-measures}

The obstruction in
Section~\ref{sec:spectral-lattice-induction} is specific to point
processes. For positive random measures, high-pass convolution
creates spectral gaps of arbitrary size while preserving all the
information contained in the return-time process.

Let $\mathcal M_+(\mathbb R^d)$ be the space of locally finite
positive Radon measures. Let $\zeta$ be a translation-invariant
Borel probability measure on this space with local second moments.
For $f\in\mathcal S(\mathbb R^d)$, set
\[
    N_f(\omega):=\omega(f).
\]
The Bartlett spectrum $\sigma_\zeta$ is determined by
\[
    \operatorname{Var}_\zeta(N_f)
    =
    \int_{\mathbb R^d}
    |\widehat f(\xi)|^2\,d\sigma_\zeta(\xi),
    \qquad
    f\in\mathcal S(\mathbb R^d).
\]
We call $\zeta$ \emph{stealthy} if $\sigma_\zeta$ vanishes on a
neighborhood of the origin.

\begin{theorem}[High-pass positive realization]
\label{thm:positive-random-measure-realization}
Let $(X,\mu,Y)$ be a separated cross-section system. For every
$R>0$, there is an equivariant Borel map
$\Phi_R:X\to\mathcal M_+(\mathbb R^d)$ such that
\[
    \Phi_R(x)=g_x\,\lambda_d,
    \qquad
    0<
    \inf_{\substack{x\in X\\u\in\mathbb R^d}}
    g_x(u),
    \qquad
    \sup_{\substack{x\in X\\u\in\mathbb R^d}}
    |\partial^\alpha g_x(u)|
    <\infty
    \quad
    (\alpha\in\mathbb N_0^d).
\]
The pushforward
$\zeta_R:=(\Phi_R)_*\mu$ has local second moments and satisfies
\[
    \sigma_{\zeta_R}\bigl(B_R(0)\bigr)=0.
\]
Moreover,
\[
    \Phi_R(x)=\Phi_R(x')
    \quad\Longleftrightarrow\quad
    \kappa_Y(x)=\kappa_Y(x').
\]
Consequently, if $Y$ is generating, then $\Phi_R$ is a measurable
$\mathbb R^d$-isomorphism from $(X,\mu)$ to
$\bigl(\mathcal M_+(\mathbb R^d),\zeta_R\bigr)$. The measure
$\zeta_R$ is ergodic whenever $\mu$ is ergodic.
\end{theorem}

\begin{proof}
Define
\[
    \vartheta(t)
    :=
    \begin{cases}
        0, & t\leq 0,\\
        e^{-1/t}, & t>0,
    \end{cases}
\]
and set
\[
    m_R(\xi)
    :=
    e^{-\|\xi\|^2}
    \vartheta\bigl(\|\xi\|^2-R^2\bigr).
\]
Then $m_R\in\mathcal S(\mathbb R^d)$ is real and even, and its
zero set is exactly $\overline{B_R(0)}$. Let $\kappa_R$ be its
inverse Fourier transform. Then $\kappa_R$ is real, even, and
Schwartz, while
\[
    \int_{\mathbb R^d}\kappa_R\,d\lambda_d
    =
    m_R(0)
    =
    0.
\]

Write $P_x:=Y_x$. Since the cross-section is separated, the sets
$P_x$ are uniformly separated. Annular counting and the rapid
decay of $\kappa_R$ give, for every multi-index $\alpha$,
\[
    C_\alpha
    :=
    \sup_{x\in X}
    \sup_{u\in\mathbb R^d}
    \sum_{p\in P_x}
    \bigl|
        \partial^\alpha\kappa_R(u-p)
    \bigr|
    <\infty.
\]
Choose $A_R>C_0$ and define
\[
    g_x(u)
    :=
    A_R+\sum_{p\in P_x}\kappa_R(u-p),
    \qquad
    \Phi_R(x):=g_x\,\lambda_d.
\]
The preceding estimates show that the functions $g_x$ are
strictly positive and smooth, with all derivatives uniformly
bounded. Their uniform boundedness also gives local second moments
for $\zeta_R$.

For $f\in C_c(\mathbb R^d)$,
\[
    \Phi_R(x)(f)
    =
    A_R\int_{\mathbb R^d}f\,d\lambda_d
    +
    \sum_{p\in P_x}(f*\kappa_R)(p).
\]
The last sum converges uniformly over the uniformly separated sets
$P_x$ and is the limit of Borel finite sums. Hence $\Phi_R$ is
Borel. The identity $P_{t.x}=P_x-t$ gives
\[
    g_{t.x}(u)=g_x(u+t),
    \qquad
    \Phi_R(t.x)=t.\Phi_R(x),
\]
so $\Phi_R$ is equivariant. Invariance and ergodicity of the
pushforward follow.

For $f\in\mathcal S(\mathbb R^d)$, evenness of $\kappa_R$ gives
\[
    \Phi_R(x)(f)
    =
    A_R\int_{\mathbb R^d}f\,d\lambda_d
    +
    N_{f*\kappa_R}^Y(x).
\]
Since $\int\kappa_R\,d\lambda_d=0$, the second term has mean zero.
The Bartlett identity for the return-time process therefore yields
\[
\begin{aligned}
    \operatorname{Var}_\mu
    \bigl(\Phi_R(\,\cdot\,)(f)\bigr)
    &=
    \int_{\mathbb R^d}
    |\widehat f(\xi)|^2
    |m_R(\xi)|^2\,d\sigma_Y(\xi).
\end{aligned}
\]
Consequently,
\[
    \sigma_{\zeta_R}=|m_R|^2\sigma_Y,
    \qquad
    \sigma_{\zeta_R}\bigl(B_R(0)\bigr)=0.
\]

It remains to determine the fibers of $\Phi_R$. If
$\Phi_R(x)=\Phi_R(x')$, equality of the corresponding densities
gives
\[
    \kappa_R*
    \bigl(
        \delta_{P_x}-\delta_{P_{x'}}
    \bigr)
    =
    0.
\]
The distribution
$\tau:=\delta_{P_x}-\delta_{P_{x'}}$ is tempered, and taking
Fourier transforms gives
\[
    m_R\widehat\tau=0,
    \qquad
    \operatorname{supp}\widehat\tau
    \subset\overline{B_R(0)}.
\]
By the Paley--Wiener--Schwartz theorem,
$\tau$ is represented on $\mathbb R^d$ by the restriction of an
entire function
\cite[Theorem~7.3.1]{Hor90}. On the other hand, its support is
contained in the locally finite set $P_x\cup P_{x'}$, so it
vanishes on a nonempty open ball. Its analytic representative
therefore vanishes identically, and
\[
    \delta_{P_x}=\delta_{P_{x'}},
    \qquad
    \kappa_Y(x)=\kappa_Y(x').
\]
The reverse implication follows immediately from the definition
of $\Phi_R$.

If $Y$ is generating, then $\kappa_Y$ is injective on an invariant
conull Borel set, and the fiber identity shows that $\Phi_R$ is
injective there as well. The Lusin--Souslin theorem provides a
Borel inverse on its image, completing the proof.
\end{proof}

\begin{remark}[Mixing positive random measures]
\label{rem:mixing-positive-random-measures}
Stealthiness does not obstruct mixing for positive random
measures. If the original $\mathbb R^d$-action is mixing, then
the factor $\zeta_R$ constructed above is mixing. If the chosen
cross-section is generating, the construction is measurably
isomorphic to the original action and therefore preserves mixing.

For a concrete example, start from a Mat\'ern type-I hard-core
process: retain precisely those points of a homogeneous Poisson
process which have no other Poisson point within distance
$r>0$. Restrictions of the retained process to sets at distance
greater than $2r$ depend on disjoint restrictions of the
underlying Poisson process and are therefore independent. Thus
the process is finite-range dependent, and in particular
mixing. Its canonical cross-section is separated and
generating. Theorem~\ref{thm:positive-random-measure-realization}
then gives, for every $R>0$, a mixing random measure whose
Bartlett spectrum vanishes on $B_R(0)$.
\end{remark}



\section{Compact open subgroups and stealthy cross-sections}
\label{sec:compact-open-groups}

We briefly leave the Euclidean setting. Let $G$ be a
second-countable locally compact abelian group with Haar measure
$m_G$ and dual group $\widehat G$. For a probability-preserving
Borel $G$-space $(X,\mu)$, write
\[
    U_gF(x):=F((-g).x),
    \qquad
    \langle U_gF,F\rangle
    =
    \int_{\widehat G}\chi(g)\,d\varsigma_F(\chi).
\]
Thus $\varsigma_F$ denotes the spectral measure of
$F\in L^2(X,\mu)$.

Let $Y\subset X$ be a cross-section for which
\[
    N_f^Y(x):=\sum_{t\in Y_x}f(t)
\]
belongs to $L^2(X,\mu)$ for every $f\in C_c(G)$. We call $Y$
\emph{stealthy} if there is an open neighborhood $U$ of the
trivial character such that
\[
    \varsigma_{(N_f^Y)^\circ}(U)=0
    \qquad
    (f\in C_c(G)).
\]
For $G=\mathbb R^d$, this agrees with the definition in
Section~\ref{sec:cross-sections} by the Bartlett--Koopman identity.

\begin{theorem}[Compact open subgroups]
\label{thm:compact-open-cross-sections}
Let $(X,\mu)$ be an essentially free probability-preserving Borel
$G$-space.
\begin{enumerate}
\item
If $G$ has a compact open subgroup, then $(X,\mu)$ admits a
stealthy cross-section.

\item
If $G=\mathbb Q_p$, then $(X,\mu)$ admits a generating stealthy
cross-section.
\end{enumerate}
\end{theorem}

\begin{proof}
After restriction to an invariant conull Borel set, we may assume
that the action is free.

Recall that a Borel section $Y$ is \emph{$U$-lacunary}, for a
neighborhood $U$ of the identity, if
\[
    Y_y\cap U=\{0\}
    \qquad
    (y\in Y).
\]
Kechris proved that free Borel actions of locally compact
second-countable groups admit complete $U$-lacunary Borel
sections \cite{Kec92}.

Let $K<G$ be compact and open. Apply the complete
lacunary-section theorem of Kechris \cite{Kec92} to the
restricted free Borel action of $K$ on $X$, taking the
lacunarity neighborhood to be all of $K$. Let $Y\subset X$ be
the resulting complete $K$-lacunary Borel section. Completeness
means that $Y$ meets every $K$-orbit, while $K$-lacunarity
prevents it from containing two distinct points in the same
$K$-orbit. Hence $Y$ meets every $K$-orbit exactly once.

Since $G/K$ is discrete, $Y_x$ is locally finite, and $N_f^Y$ is
uniformly bounded for each fixed $f\in C_c(G)$. For every $x\in X$,
the return-time set $Y_x$ therefore meets each coset of $K$ in
exactly one point. Since $G/K$ is discrete, $Y_x$ is locally
finite, and $N_f^Y$ is uniformly bounded for each fixed
$f\in C_c(G)$.

For such an $f$, integration over $K$ gives
\[
\begin{aligned}
    \int_K U_kN_f^Y(x)\,dm_G(k)
    &=
    \sum_{t\in Y_x}\int_K f(t+k)\,dm_G(k)\\
    &=
    \int_G f\,dm_G,
\end{aligned}
\]
because the translates $t+K$, $t\in Y_x$, partition $G$. Taking
expectations and using invariance of $\mu$ gives
\[
    m_G(K)
    \int_XN_f^Y\,d\mu
    =
    \int_Gf\,dm_G.
\]
Therefore
\[
\begin{aligned}
    \int_K
    U_k(N_f^Y)^\circ(x)\,dm_G(k)
    &=
    \int_Gf\,dm_G
    -
    m_G(K)
    \int_XN_f^Y\,d\mu\\
    &=0.
\end{aligned}
\]
The spectral theorem gives
\[
    d\varsigma_{\int_KU_k(N_f^Y)^\circ\,dm_G(k)}(\chi)
    =
    \left|
        \int_K\chi(k)\,dm_G(k)
    \right|^2
    d\varsigma_{(N_f^Y)^\circ}(\chi).
\]
If
\[
    K^\perp
    :=
    \{\chi\in\widehat G:\chi|_K=1\},
\]
then
\[
    \int_K\chi(k)\,dm_G(k)
    =
    m_G(K)\mathbf 1_{K^\perp}(\chi).
\]
It follows that
\[
    \varsigma_{(N_f^Y)^\circ}(K^\perp)=0
    \qquad
    (f\in C_c(G)).
\]
Since $K^\perp$ is an open neighborhood of the trivial character,
$Y$ is stealthy. This proves the first assertion.

Now let $G=\mathbb Q_p$, put $K=\mathbb Z_p$ and
$a=p^{-1}$, and choose a Borel transversal $Z\subset X$ for the
$K$-orbits. Freeness makes
\[
    K\times Z\longrightarrow X,
    \qquad
    (k,z)\longmapsto k.z,
\]
a Borel bijection. Define Borel maps $T:Z\to Z$ and $A:Z\to K$ by
\[
    a.z=A(z).Tz.
\]
Iterating this identity \(p\) times and using \(pa=1\in K\) gives
\[
    T^pz=z,
    \qquad
    \sum_{j=0}^{p-1}A(T^jz)=1.
\]
Moreover, $T$ has no orbit of length strictly between \(1\) and
\(p\). Indeed, if $T^jz=z$ for some $1\leq j<p$, then iteration
and freeness would imply $j/p\in\mathbb Z_p$, which is impossible.
Thus every $T$-orbit has exactly \(p\) elements.

Consider
\[
    V
    :=
    \left\{
        (v_0,\ldots,v_{p-1})\in K^p:
        \sum_{j=0}^{p-1}v_j=1
    \right\},
    \qquad
    \theta(v_0,\ldots,v_{p-1})
    :=
    (v_1,\ldots,v_{p-1},v_0).
\]
The action generated by $\theta$ is free. Since \(p\) is prime,
a vector fixed by a nontrivial power of $\theta$ would be constant,
and its coordinates would satisfy \(pv_0=1\), which has no
solution in \(\mathbb Z_p\).

Choose Borel transversals $C\subset Z$ and $W\subset V$ for the
actions generated by $T$ and $\theta$. The space $W$ is an
uncountable standard Borel space, so there is a Borel injection
$\iota:C\to W$. Extend it equivariantly by
\[
    B(T^jc):=\theta^j\iota(c),
    \qquad
    c\in C,\quad 0\leq j<p.
\]
Then $B:Z\to V$ is Borel and injective, and
$B(Tz)=\theta B(z)$. Writing
\[
    B(z)
    =
    \bigl(
        \beta(z),\beta(Tz),\ldots,\beta(T^{p-1}z)
    \bigr),
\]
we have
\[
    \sum_{j=0}^{p-1}\beta(T^jz)
    =
    \sum_{j=0}^{p-1}A(T^jz)
    =
    1.
\]
Hence $\beta-A$ has zero sum on every \(T\)-orbit and is a Borel
coboundary. Explicitly, set $\phi(c)=0$ on \(C\) and
\[
    \phi(T^jc)
    :=
    -\sum_{i=0}^{j-1}
    (\beta-A)(T^ic),
    \qquad
    1\leq j<p.
\]
Then
\[
    \beta=A+\phi-\phi\circ T.
\]

Define
\[
    y(z):=\phi(z).z,
    \qquad
    Y:=y(Z).
\]
The map $y$ is Borel and injective, since two points with the same
image must lie in the same \(K\)-orbit and \(Z\) is a transversal.
Thus $Y$ is Borel by the Lusin--Souslin theorem. It again meets
every \(K\)-orbit exactly once, and the preceding coboundary
identity gives
\[
    a.y(z)=\beta(z).y(Tz).
\]
Consequently, \(Y\) is stealthy by the first part of the proof.

It remains to prove that it is generating. Write
$x=u_0.y(z)$ with \(u_0\in K\). For \(0\leq j<p\), let
$r_j(x)$ be the unique point of \(Y_x\cap(ja+K)\) and put
\(u_j:=ja-r_j(x)\). Iterating the identity for \(a.y(z)\) gives
\[
\begin{aligned}
    u_{j+1}-u_j
    &=
    \beta(T^jz),
    &&0\leq j<p-1,\\
    u_0+1-u_{p-1}
    &=
    \beta(T^{p-1}z).
\end{aligned}
\]
The return-time set therefore determines
\[
    B(z)
    =
    \bigl(
        u_1-u_0,\ldots,
        u_{p-1}-u_{p-2},
        u_0+1-u_{p-1}
    \bigr).
\]
Only the $p$ cosets
\[
    ja+\mathbb Z_p,
    \qquad
    0\leq j<p,
\]
enter this reconstruction. The argument does not require
$a+\mathbb Z_p$ to generate the quotient
$\mathbb Q_p/\mathbb Z_p$.

Since $B$ is injective, it determines \(z\). It also determines
\(u_0=-r_0(x)\), and hence \(x=u_0.y(z)\). The return-time map is
therefore injective, so \(Y\) is generating.
\end{proof}


\section*{Statements and Declarations}

\paragraph{\textbf{Funding.}}
This work was supported by the Swedish Research Council under grant
VR 11253322.

\paragraph{\textbf{Competing interests.}}
The author has no relevant financial or non-financial interests to disclose.

\paragraph{\textbf{Data availability.}}
No datasets were generated or analysed during the current study.


\end{document}